\documentclass{amsart}
\usepackage{amssymb}
\usepackage{amsmath}
\usepackage{amsthm}
\usepackage{mathrsfs}
\usepackage{hyperref}

\title[Weak capacity and critical exponents]{Weak capacity and critical exponents} 
\author{Jeff Lindquist}
\email{jeffrey.lindquist@helsinki.fi}
\address{PL 68 (Gustaf H\"allstr\"omin katu 2b) 00014 Helsingin Yliopisto}
\thanks{At the University of Helsinki, the author was partially supported by Academy of Finland grants 297258 and 308759. At the University of California, Los Angeles, the author was partially supported by NSF grants DMS-1506099 and DMS-1162471.}

\newtheorem{thm}{Theorem}[section]
\newtheorem{lemma}[thm]{Lemma}
\newtheorem{cor}[thm]{Corollary}
\newtheorem{prop}[thm]{Proposition}
\theoremstyle{definition}
\newtheorem{defn}[thm]{Definition}
\newtheorem*{ack}{Acknowledgements}
\newtheorem{remark}[thm]{Remark}

\newcommand{\N}{\mathbb{N}}

\newcommand{\R}{\mathbb{R}}

\newcommand{\mH}{\mathscr{H}}

\newcommand{\norm}[1]{\left\lVert#1\right\rVert}
\newcommand{\wpnorm}[1]{\left\lVert#1\right\rVert_{p,\infty}}
\newcommand{\Stwo}{\mathbb{S}^2}

\DeclareMathOperator{\diam}{diam}
\DeclareMathOperator{\ARCdim}{ARCdim}
\DeclareMathOperator{\dist}{dist}

\DeclareMathOperator{\modtwo}{mod_2}
\DeclareMathOperator{\qcap}{wcap_\mathnormal{Q}}
\DeclareMathOperator{\capq}{wcap_\mathnormal{q}}
\DeclareMathOperator{\pcap}{wcap_\mathnormal{p}}
\DeclareMathOperator{\wcap}{wcap}
\DeclareMathOperator{\wcaptwo}{wcap_2}
\DeclareMathOperator{\qmod}{mod_\mathnormal{Q}}
\DeclareMathOperator{\qmodg}{mod_\mathnormal{Q}^G}

\DeclareMathOperator{\modp}{mod_\mathnormal{p}}

\DeclareMathOperator{\Modp}{Mod_\mathnormal{p}}
\DeclareMathOperator{\Iso}{Iso}

\def\XXint#1#2#3{{\setbox0=\hbox{$#1{#2#3}{\int}$ }
\vcenter{\hbox{$#2#3$ }}\kern-.6\wd0}}

\begin{document}
\maketitle

\begin{abstract} 
We investigate critical exponents relating to weak capacity in Ahlfors regular metric measure spaces.  This allows a proof of a weak capacity version of a result by Bonk and Kleiner about the uniformization of metric $2$-spheres.  Using our result, we promote local quasisymmetric equivalence with $\Stwo$ to global quasisymmetric equivalence.  We also use our weak capacity version to derive conditions for quasisymmetric equivalence to $\Stwo$ in the presence of a group action.  We investigate the relation between our defined critical exponents and Ahlfors regular conformal dimension, particularly in the cases where the Combinatorial Loewner Property is present or the space attains its Ahlfors regular conformal dimension.
\end{abstract}
\section{Introduction}

In this paper we expand upon on work in \cite{L1} to study compact, connected, Ahlfors regular metric measure spaces and quasisymmetric maps between these spaces.  The main tool we use is weak $p$-capacity, denoted $\pcap$.  Weak capacity is a generalization of the modulus of path families between sets that has a quasisymmetric robustness property.  The rough idea is to use a hyperbolic filling of a metric space and mirror the definition of path family modulus by using paths connecting sets in a hyperbolic filling as opposed to those in the space itself.  

Ahlfors regular metric spaces and quasisymmetric maps arise readily in the study of Gromov hyperbolic groups.  Gromov hyperbolic groups are finitely generated groups whose Cayley graphs are Gromov hyperbolic metric spaces (see \cite{GH}).  The boundary at infinity of a Cayley graphs are Ahlfors regular metric spaces when equipped with a visual metric.  These boundaries depend on the generating set, but as metric spaces they are quasisymmetrically equivalent.

Hyperbolic fillings are graphs that are analogous to Cayley graphs in the hyperbolic group setting.  The difference is that in the group case we are first given the Cayley graph and construct the boundary, whereas in our setting we start with the ``boundary'' metric space and construct the graph.

Our first result is a generalization of a result by Bonk and Kleiner in \cite{BnK} to the weak capacity setting.  In \cite{BnK}, the question of when a metric space homeomorphic to the $2$-sphere $\Stwo$ is quasisymmetrically equivalent to $\Stwo$ is investigated. Theorem \ref{qs unif thm} provides a characterization of when a compact, connected, $Q$-regular metric measure space is quasisymmetrically equivalent to $\Stwo$.  This characterization is in terms of positively separated open sets, which are open sets $A,B \subseteq Z$ with $\dist(A,B) > 0$.  To state Theorem \ref{qs unif thm}, we use a function $\varphi_2$ which provides control of the weak capacity of two positively separated open sets in terms of their relative distance; that is, for positively separated open sets $A,B$ in $Z$, we have $\wcaptwo(A,B) \leq \varphi_2(\Delta(A,B))$, where $\wcaptwo$ is the weak $2$-capacity and $\Delta(A,B)$ is the relative distance between $A$ and $B$.    We leave the formal definitions to Section \ref{Section preliminaries}.  We also require the condition that our space be linearly locally connected.  This is necessary for the existence of a quasisymmetric map to $\Stwo$.

\begin{thm}\label{qs unif thm}
Let $(Z,d, \mu)$ be an Ahlfors $Q$-regular metric measure space homeomorphic to $\mathbb{S}^2$ that is linearly locally connected.  Suppose there exists $t_0 > 0$ such that $\varphi_2(t_0) < \infty$.  Then, $Z$ is quasisymmetrically equivalent to $\Stwo$.  Conversely, if $Z$ is quasisymmetrically equivalent to $\Stwo$, then $\varphi_2(t)$ is finite for all $t$ and $\varphi_2(t) \to 0$ as $t \to \infty$.
\end{thm}

From this, we arrive at another proof of a special case of a result in \cite{Su} about metric $2$-spheres which are locally quasisymmetric to subsets of $\Stwo$.  This was also proven in \cite{GW} by creating a quasisymmetric atlas.

\begin{cor} \label{local QS}
Let $(Z,d, \mu)$ be an Ahlfors $Q$-regular metric measure space homeomorphic to $\mathbb{S}^2$ that is linearly locally connected. Suppose for every $z \in Z$ there is a neighborhood $U_z$ which is quasisymmetrically equivalent to a subset of $\Stwo$.  Then, $Z$ is quasisymmetrically equivalent to $\Stwo$.
\end{cor} 

In the presence of a group action, we do not need to control the weak capacity of two sets by their relative distance a priori.   Instead, it is sufficient to require that all pairs of open sets that are at positive distance to one another have finite weak capacity.  This is captured by the critical exponents $Q_w(Z)$ and $Q_w'(Z)$. The quantity $Q_w$ is the infimum of $Q>1$ such that if $p>Q$, then for all positively separated open sets, $\pcap(A,B) < \infty$.  The quantity $Q_w'$ is the infimum of $Q>1$ such that if $p > Q$, then for all positively separated open sets, $\pcap(A,B) \leq \varphi_p(\Delta(A,B))$ where $\varphi_p (t) < \infty$ for all $t$.

\begin{thm} \label{group action thm}
Let $(Z,d,\mu)$ be a compact, connected, Ahlfors regular metric measure space.  Suppose there is a metric space $Y$ quasi-isometric to a hyperbolic filling $X$ of $Z$ and a group $G \subseteq \Iso(Y)$ acting cocompactly on $Y$.  Then, $Q_w(Z) = Q_w'(Z)$.  If the infimum is attained in the definition of $Q_w$, then the infimum is attained in the definition of $Q_w'$.
\end{thm}

Here we use the notation $\Iso(Y)$ for the group of isometries of $Y$.  Recall that a group action on $Y$ is {\em cocompact} if there is a compact subset $K \subseteq Y$ such that $Y = \cup_{g \in G} gK$.  Equivalently, there is a compact subset $K \subseteq Y$ such that for every point $y \in Y$, we can find a group element $g$ such that $gy \in K$.

The proof of Theorem \ref{group action thm} shows that in the presence of a cocompact group action we can further reduce the conditions in Theorem \ref{qs unif thm} to only needing to check finitely many sets.

\begin{cor}\label{QS unif cor}
Let $(Z,d, \mu)$ be an Ahlfors $Q$-regular metric measure space homeomorphic to $\Stwo$ that is linearly locally connected.  Suppose there is a metric space $Y$ quasi-isometric to a hyperbolic filling $X$ of $Z$ and a group $G \subseteq \Iso(Y)$ acting cocompactly on $Y$.  Then, there exist finitely many pairs of sets $(A_i, B_i)$ with the property that if $\wcaptwo(A_i, B_i) < \infty$ for all $i$, then $Z$ is quasisymmetrically equivalent to $\Stwo$.
\end{cor}

Our next theorem concerns metric spaces which satisfy the Combinatorial Loewner Property developed in \cite{BdK}.  This property requires that the discrete modulus of two sets is eventually bounded both above and below by functions of the relative distance.  For precise definitions, we refer the reader to Section \ref{Section other thms}.

\begin{thm} \label{CLP theorem}
Let $(Z,d,\mu)$ be an Ahlfors regular metric space that satisfies the Combinatorial Loewner Property with exponent $Q_{CLP}(Z) \geq 1$.  Then, for all $1 \leq p < Q_{CLP}$ and positively separated open sets $A,B \subseteq Z$ we have $\pcap(A,B) = \infty$.
\end{thm}

This theorem is analogous to Remark $1$ at the end of Section $3$ in \cite{BdK} which states that by a result in \cite{HP}, given disjoint continua $A,B \subseteq Z$ and $p < Q_{CLP}$ we have $\lim_{k \to \infty} \Modp(A,B,G_k) = \infty$.  Here $\Modp(A,B,G_k)$ is the discrete modulus of paths connecting $A$ and $B$ in the discretization $G_k$ of $Z$.  Theorem \ref{CLP theorem} shows $Q_{CLP} \leq Q_w$.

Our last theorem is concerned with Ahlfors regular conformal dimension $\ARCdim(Z)$ of a metric space $Z$.  This is the infimum of the dimensions of metric spaces that are both Ahlfors regular and quasisymmetrically equivalent to $Z$.  Our theorem shows that if $Z$ attains it Ahlfors regular conformal dimension, then the quantities $Q_w'(Z)$, and $\ARCdim(Z)$ agree.

\begin{thm} \label{Wk Tan Thm}
Let $(Z,d, \mu)$ be an Ahlfors $Q$-regular metric measure space with $\ARCdim(Z) = Q$.  Let $p < Q$.  Then, there exists a sequence of positively separated open sets $A_k, B_k$ and a constant $C > 0$ such that $\Delta(A_k, B_k) < C$, but $\pcap(A_k, B_k) \to \infty$ as $k \to \infty$.
\end{thm}

This shows for $p < Q$ that $p \leq Q_w'$, so $Q_w' \geq Q = \ARCdim$.  The proof of this result relies on the existence of a family of curves with positive $Q$-modulus in a weak tangent of $(Z,d)$.  This was first proven in \cite{KL}, but we will use a version provided by \cite{CP}.  

We now outline the remainder of the paper.  In Section \ref{Section preliminaries}, we provide precise definitions to concepts discussed above.  In Section \ref{Section wcap basics}, we state previous results from \cite{L1} and prove properties of weak capacity.  In Section \ref{Section control functions}, we prove path lemmas as well as investigate properties of $\varphi_p$.  In Section \ref{Section qs unif}, we prove Theorem \ref{qs unif thm} as well as its corollaries.  In Section \ref{Section other thms}, we prove Theorems \ref{CLP theorem} and \ref{Wk Tan Thm}.

Given quantities $A,B$ we will sometimes use the symbol $\lesssim$ to mean there is a constant $C > 0$ depending on some ambient parameters such that $A \leq C B$.  We will likewise use $\gtrsim$ to mean there is a constant $C>0$ depending on some ambient parameters such that $CA \geq B$.  If $A \lesssim B$ and $B \lesssim A$, we may write $A \simeq B$.  We use $a \vee b = \max(a,b)$ and $a \wedge b = \min(a,b)$.


\begin{ack}
The author thanks Mario Bonk introducting the subject material and frequently discussing the research in this paper; his guidance has been extremely valuable.  The author thanks Pekka Pankka and Eero Saksman for their read-throughs of drafts of this paper and many helpful comments.  The author thanks both UCLA and the University of Helsinki for their support.
\end{ack}


\section{Preliminary definitions and concepts}\label{Section preliminaries}

\subsection{General concepts}

First we fix some notation.  Given a point $z$ in a metric space $(Z,d)$ and a real number $r \geq 0$, we set ${B(z,r) = \{z' \in Z : d(z,z') < r\}}$.  Given $A \subseteq Z$, we write $\overline{A}$ for the closure of $A$ in $Z$.  Thus, $\overline{B(z,r)}$ is the closure of $B(z,r)$ in $Z$, which is contained in (but not necessarily equal to) the closed ball $\overline{B}(z,r) = \{z' \in Z : d(z,z') \leq r\}$.  Given two sets $A,B \subseteq Z$, we write $\dist(A,B) = \inf \{d(a,b) : a \in A, b \in B\}$ for the {\em distance} between $A$ and $B$.  We set $\diam(A) = \sup\{d(a,a') : a, a' \in A\}$.  In what follows, we will often assume $\diam(Z) = 1$, which can be attained by rescaling $d$ as we will be working with compact metric spaces.  

A metric measure space $(Z,d,\mu)$ (so $\mu$ is a Borel regular measure) is called {\em Ahlfors $Q$-regular} if there exist constants $c, C>0$ such that for all $r \leq \diam(Z)$ we have $cr^Q \leq \mu(B(z,r)) \leq Cr^Q$.  In this case, it follows that the Hausdorff $Q$-measure $\mH^Q$ also has this property, so we may also refer to just a metric space $(Z,d)$ as being Ahlfors $Q$-regular, or just $Q$-regular.  We also may de-emphasize the $Q$ and refer to regular or Ahlfors regular metric spaces.

Given $C \geq 1$, we say a metric space $(Z,d)$ is {\em $C$-uniformly perfect} if for every $z \in Z$ and every $r < \diam(Z)$, the set $B(z,r) \setminus B(z, r/C)$ is nonempty.  This condition is similar to connectivity; it forbids isolated points and all connected spaces are uniformly perfect.  A metric space $(Z,d)$ is {\em doubling} if there exists a constant $N>0$ such that for every $z \in Z$ and $r > 0$, there are at most $N$ points $\{z_i\}$ such that $B(z, r) \subseteq \cup_i B(z_i, r/2)$.  It follows that there is a function $N \colon [1,\infty) \to \N$ such that for every $z \in Z$ and $r>0$, there are at most $N(\alpha)$ points $\{z_i\}$ such that $B(z,r) \subseteq \cup_i B(z_i, r/\alpha)$.  Both of these properties hold in Ahlfors regular metric spaces from straightforward measure computations.

\sloppy
The maps between these metric spaces that we are primarily interested in are quasisymmetric homeomorphisms.  Given an increasing function $\eta \colon (0,\infty) \to (0,\infty)$ with $\lim_{t \to 0} \eta(t) = 0$, a homeomorphism $f \colon (Z,d) \to (Z', d')$ is called an {\em $\eta$-quasisymmetry} ($\eta$-qs) if for all $x, y, z \in Z$ with $x \neq z$ we have
\begin{equation*}
\frac{d'(f(x),f(y))}{d'(f(x),f(z))} \leq \eta\biggl(\frac{d(x,y)}{d(x,z)}\biggr).
\end{equation*}
Quasisymmetric maps are global metric generalizations of quasiconformal maps in the sense that they control the distortion of the relative shapes of sets.  If there is a quasisymmetry $f \colon Z \to Z'$ between metric spaces $(Z,d)$ and $(Z',d')$, we say these metric spaces are {\em quasisymmetrically equivalent}.

Quasisymmetric maps correspond to quasi-isometric maps on the hyperbolic filling.  Given metric spaces $(X, d_X)$ and $(Y, d_Y)$ and real numbers $C,D > 0$, a function $f \colon X \to Y$ is called a {\em quasi-isometry} if for all $x, x' \in X$ we have 
\begin{equation*}
\frac{1}{C}d_X(x, x') - D \leq d_Y(f(x), f(x')) \leq Cd_X(x, x') + D
\end{equation*}
and for every point $y \in Y$, there is an $x \in X$ with $d_Y(f(x), y) \leq D$.

For Theorem \ref{qs unif thm}, we require the notion of linear local connectivity.

\begin{defn}
A metric space $(Z,d)$ is $\lambda$-LLC (linearly locally connected) for $\lambda \geq 1$ if the following two conditions hold: \\
\indent (i) Given $B(a,r) \subseteq Z$ and $x,y \in B(a,r)$, there exists a continuum $E \subseteq B(a, \lambda r)$ containing $x$ and $y$. \\
\indent (ii) Given $B(a,r) \subseteq Z$ and $x,y \in Z \setminus B(a,r)$, there exists a continuum $E \subseteq Z \setminus B(a, r/\lambda)$ containing $x$ and $y$. \\
If $(Z,d)$ is $\lambda$-LLC for some $\lambda \geq 1$, we say $(Z,d)$ is LLC.
\end{defn}

A {\em continuum} is a compact, connected set with more than one point.  Heuristically, the LLC condition forbids a space from having very sharp cusps.  These pose an obstruction for quasisymmetric deformation: if one metric space is LLC and is quasisymmetrically equivalent to another metric space, then that other metric space must also be LLC (for a possibly different $\lambda$).  

\fussy

\subsection{Hyperbolic fillings and paths}
To define a hyperbolic filling $X = (V_X, E_X)$ for our compact, connected, $Q$-regular metric space $(Z,d)$ we first fix a real number $s>0$ called the {\em scaling parameter}.  Then, for each $k \in \N$, we choose an $s^{-k}$-separated set $Z_k$ that is maximal relative to inclusion.  To each $z_i^k \in Z_k$, we define a vertex $v_i^k$ associated to the ball $B(z_i^k, 2s^{-k})$.  By abuse of notation, we will often consider $v \in Z$ as the center of the vertex $v$ and write the ball associated to $v$ as $B(v, 2s^{-k})$ or $B_v$.  We write $r(B_v) = r(v) = 2s^{-k}$ for the radius of $B_v$.  We call $k$ the {\em level} of the vertex $v$ and write this as $\ell(v)$, so $r(B_v) = 2s^{-\ell(v)}$.  As $\diam(Z) = 1$, for $k=0$ we define a unique vertex $O$ associated to the entire set $Z$.  For $k \geq 1$, we define the vertices on level $k$ as $V_X^k = \cup_i v_i^k$ and we set $V_X^0 = O$.  We set $V_X = \cup_k V_X^k$.  To define $E_X$, we connect distinct vertices $v,w \in V_X$ by an edge if and only if $|\ell(v) - \ell(w)| \leq 1$ and $B_v \cap B_w \neq \emptyset$.  It follows from the existence of $O$ that $X$ is a connected graph.  We equip $X$ with the graph metric (denoted by $|\cdot - \cdot|$) which identifies each edge isometrically with the interval $[0,1]$ in $\R$.  This means if $v,w$ are vertices in $V_X$ then $|v-w|$ is the smallest number of edges required to connect $v$ to $w$.

This construction has a number of nice properties (for proofs and definitions, see \cite[Section 3]{L1}).  The filling $X$ is geodesic and Gromov hyperbolic and the original space $Z$ can be recovered as the boundary at infinity $\partial_\infty X$ equipped with a visual metric.  As $Z$ is doubling, $X$ has bounded degree.  Thus, the hyperbolic filling functions much as the Cayley graph of a hyperbolic group.  While our construction relies on some choices, specifically the choice of $s$ and the choices of the sets $Z_k$, different choices result in quasi-isometrically equivalent graphs (in fact, the vertex sets are bilipschitz equivalent; see \cite{L2}).  Distance in $X$ will typically be denoted with the notation $|\cdot - \cdot|$, but sometimes we will use $\rho(\cdot, \cdot)$ as well, especially when dealing with the distance between subsets.

We now specify what we mean by paths in $X$.  A {\em vertex path} in $X$ is a (possibly two-sided infinite, one-sided infinite, or finite) sequence of vertices $\{v_k\}$ such that for all $k$ for which the vertices $v_k$ and $v_{k+1}$ exist, these vertices are connected by an edge.  Similarly, an {\em edge path} in $X$ is a sequence of edges $\{e_k\}$ such that for all $k$ for which the edges $e_k$ and $e_{k+1}$ exist, these edges share a common vertex.  There is a bijection between vertex paths and edge paths.  Given a vertex path $\{v_k\}$, we define a corresponding edge path $\{e_k\}$ by setting $e_k$ to be the edge connecting $v_k$ to $v_{k+1}$.  Given an edge path $\{e_k\}$, we define a corresponding vertex path $\{v_k\}$ by setting $v_k$ to be the vertex that $e_{k-1}$ and $e_k$ have in common.  Hence, we can refer to both simultaneously as a {\em path}.  A one-sided infinite vertex path $\{v_k\}_{k=1}^{\infty}$ (and hence a one-sided infinite vertex path) converges to a point $z \in Z$ if $v_k \to z$ and $r(B_{v_k}) \to 0$.  A two-sided infinite vertex path $\{v_k\}$ may have at most two limits in the above sense, one as $k \to \infty$ and one as $k \to -\infty$.

\subsection{Weak norms and weak capacity}
To define weak capacity, we will use the weak norm.  Given a countable set $X$ and a function $f \colon X \to [0,\infty)$, let $\norm{f}_{p,\infty}$ be the infimum of $C > 0$ such that for all $\lambda > 0$ we have
\begin{equation} \label{weak def ineq}
\#\{x \in X : f(x) > \lambda\} \leq \frac{C^p}{\lambda^p}.
\end{equation}
Set $\ell^{p,\infty}(X) = \{f : \wpnorm{f} < \infty\}$.  The quantity $\wpnorm{\cdot}$ does not define a norm, but it is comparable to one up to a multiplicative constant when $p > 1$ (see \cite[Section 2]{BnS}).  For what we are concerned with, a fixed multiplicative constant has no impact.  Hence. we freely refer to $\wpnorm{\cdot}$ as a norm and use the comparable inequality (\ref{weak def ineq}).

Now that we have a hyperbolic filling, notions of path limits, and the weak norm, we can define weak capacity.  This is a quantity defined between two sets $A, B \subseteq Z$.  Given sets $A,B \subset Z$, we say $A$ and $B$ have {\em positive separation}, or are {\em positively separated}, if $\dist(A,B) > 0$.  For our purposes we will only need to use open sets with positive separation, so we define weak capacity only in this case (this avoids technicalities such as having to deal with nontangential limits).

Given a hyperbolic filling $X$ of $(Z,d)$ and open, positively separated sets $A,B \subseteq Z$, a function $\tau \colon E_X \to [0,\infty)$ is called {\em admissible} for $A,B$ if for all edge paths $\gamma$ with limits in $A$ and $B$ we have $\sum_{e \in \gamma} \tau(e) \geq 1$, where $\sum_{e \in \gamma} \tau(e)$ is shorthand for writing the edges in $\gamma$ as a bi-infinite sequence $\{e_k\}$ and summing $\sum_{k=-\infty}^{\infty} \tau(e_k)$.  From this definition, it is clear that we can restrict the codomain of $\tau$ to $[0,1]$ by redefining $\tau$ as $\min(\tau, 1)$.

\begin{defn}
Given a positively separated open sets  $A,B \subseteq Z$, we define the weak $p$-capacity $\pcap(A,B)$ between $A$ and $B$ as 
\begin{equation*}
\pcap(A,B) = \inf \{ \norm {\tau}_{p,\infty}^p : \tau \text{ is admissible for $A$ and $B$} \}.
\end{equation*}
\end{defn}
The ambiguity that arises from different choices of hyperbolic fillings is just a multiplicative constant; see Theorem \ref{QS inv qcap} (\cite[Theorem 1.4]{L1}).  We can use vertex functions $\tau \colon V_X \to [0,1]$ for admissibility instead of edge functions also at the cost of a multiplicative constant.  This is proven in Lemma \ref{vertex edge comp}.  Further properties of $\pcap$ are also discussed in Section \ref{Section wcap basics}.

\subsection{Critical exponents}
This paper is primarily concerned with critical exponents that arise from $\pcap$ and their relation to {\em Ahlfors regular conformal dimension}.  This is denoted by $\ARCdim(Z),$ or $\ARCdim$ and defined as follows.
Let 
\begin{equation*}
\mathscr{G}_d = \{\theta: \theta \text{ is a metric on } Z \text{ with } (Z,\theta) \sim_{qs} (Z,d)\}.
\end{equation*}
Then $\ARCdim((Z,d)) = \inf  \dim_H (Z,\theta)$ where the infimum is taken over all $\theta \in \mathscr{G}_d$ such that $(Z,\theta)$ is Ahlfors regular.
Here $(Z,\theta) \sim_{qs} (Z,d)$ means that the identity map is a quasisymmetry.  Our first critical exponent is concerned with when $\pcap$ is finite.

\begin{defn}
\label{Qw def}
\begin{equation*}
Q_w((Z,d)) = \inf\{p:\pcap(A,B) < \infty \text{ for all open, positively separated } A \text{ and }B\}.
\end{equation*}
\end{defn}
We will write $Q_w$ instead when $(Z,d)$ is understood.

Our second critical exponent requires more control over $\pcap$.  For this, we use the notion of relative distance.  Given $A,B \subseteq Z$, the {\em relative distance} of $A$ and $B$ is 
\begin{equation*}
\Delta(A,B) = \frac{\dist(A,B)}{\min(\diam(A), \diam(B))}.
\end{equation*}

\begin{defn}
Given a hyperbolic filling $X$, we set 
\begin{equation*}
\varphi_p(t) = \sup \{\pcap(A,B) : \Delta(A,B) \geq t\}
\end{equation*}
where $A$ and $B$ are open sets.  We call this a {\em control function}.
\end{defn}

\begin{defn}
\label{Qw' def}
\begin{equation*}
Q_w'((Z,d)) = \inf \{p: \text{there exists } t_0 > 0 \text{ such that } \varphi_p(t_0) < \infty\}.
\end{equation*}
\end{defn}

The definition of $Q_w'$ is stronger than it appears at first glance.  Indeed, it follows that if $\varphi_p(t_0) < \infty$ for some $t_0$, then $\varphi_p(t) < \infty$ for all $t>0$ and that $\lim_{t \to 0} \varphi_p(t) = 0$; see Lemma \ref{varphi equivalence}.  

It is clear that $Q_w \leq Q_w'$.  It will be shown in Section \ref{Section wcap basics} that $Q_w' \leq \ARCdim$.

Theorem \ref{qs unif thm} is related to $Q_w'$.  Specifically, if $\varphi_2$ is finite, then $Q_w' = 2$ (note we cannot have $Q_w' < 2$ by Theorem \ref{Wk Tan Thm} below).  Likewise, if the infimum is attained at $2$ in the definition of $Q_w'$, then $\varphi_2$ is finite.


\section{Basic properties and critical exponents}\label{Section wcap basics}

\subsection{Basic properties}
We first prove that in the definition of $\pcap$, we can use vertex functions instead of edge functions to get a comparable quantity.  For positively separated open sets $A,B \subseteq Z$, call a vertex function $\tau \colon V_X \to [0,1]$ {\em admissible} for $A$ and $B$ (or for $\pcap(A,B)$) if $\sum_{v \in \gamma} \tau(v) \geq 1$ for every path $\gamma$ in $X$ with limits in $A$ and $B$.

\begin{lemma} \label{vertex edge comp}
Let $(Z, d, \mu)$ be an Ahlfors $Q$-regular metric measure space and let $p>1$.  Fix a hyperbolic filling $X = (V_X, E_X)$.  For positively separated open sets $A,B \subseteq Z$, define $\pcap'(A,B)$ by infimizing $\wpnorm{\tau}^p$ over admissible vertex functions $\tau$ for $A$ and $B$.  Then, there are constants $c, C>0$ depending only on $Q$ and the hyperbolic filling parameters with the following property: whenever $A,B \subseteq Z$ are positively separated open sets,
\begin{equation*}
c\pcap(A,B) \leq \pcap'(A,B) \leq C\pcap(A,B).
\end{equation*}
\end{lemma}

\begin{proof}
Fix positively separated open sets $A, B \subseteq Z$.  

\sloppy
First, let $\tau$ be an admissible (edge) function for $\pcap(A,B)$.  For $v \in V_X$, set $\tau'(v) = \sum_{e \sim v} \tau(e)$, where $e \sim v$ means the edge $e$ is incident to $v$.  Let $\gamma$ be a path connecting $A$ and $B$ in $X$.  We can view $\gamma$ as an bi-infinite sequence of alternating vertices and edges $(\dots, v_{k-1}, e_{k-1}, v_k, e_k, v_{k+1}, \dots)$ where $e_j$ is incident to $v_j$ and $v_{j+1}$ for all $j$.  Then, for all $j$ we have $\tau'(v_j) \geq \tau(e_j)$.  It follows that ${\sum_{v \in \gamma} \tau'(v) \geq \sum_{e \in \gamma} \tau(e)}$, so $\tau'$ is an admissible vertex function for $\pcap'(A,B)$.  Now, $X$ has bounded degree (\cite[Lemma 3.5]{L1}), meaning there is a constant $D > 0$ such that each vertex $v$ is incident to at most $D$ edges.  Hence, if $\tau'(v) > \lambda$, then one of the edges incident to $v$ must have $\tau(e) > \lambda / D$.  Each edge is incident to exactly two vertices, so it follows that
\begin{equation*}
\#\{v : \tau'(v) > \lambda\} \leq 2\#\{e : \tau(e) > \lambda/D\} \leq 2 \wpnorm{\tau}^p D^p / \lambda^p
\end{equation*}
so $\wpnorm{\tau'}^p \leq 2D^p \wpnorm{\tau}^p$.  Hence, $\pcap'(A,B) \leq 2D^p \pcap(A,B)$.

\fussy
For the other direction, we use a similar construction.  Let $\tau'$ be an admissible (vertex) function for $\pcap'(A,B)$.  For $e \in E_X$, set $\tau(e) = \sum_{v \sim e} \tau'(v)$.  By an argument very similar to the above, it follows that $\tau$ is admissible for $\pcap(A,B)$ and that $\wpnorm{\tau'}^p \leq D2^p \wpnorm{\tau}^p$, where $D$ is a bound on the degree of $X$.
\end{proof}

For the rest of this paper, we will use either vertex functions or edge functions depending on convenience.  If a statement refers to an admissible function without specifying, it should be understood that either choice is applicable to that statement.

We now collect basic properties of $\pcap$.

\begin{lemma} \label{basic properties wcap}
Let $(Z, d, \mu)$ be an Ahlfors $Q$-regular metric measure space.  Fix a hyperbolic filling $X = (V_X, E_X)$ of $Z$. Let $A,B$ be positively separated open sets.  \\
\indent {\normalfont(i)} {\normalfont (Monotonicity)} $\pcap(A,B) \leq \pcap(A', B')$ whenever $A', B'$ are positively separated open sets such that $A \subseteq A'$ and $B \subseteq B'$.  \\
\indent {\normalfont(ii)} {\normalfont (Countable subadditivity)} Suppose $A_k$ are open sets and $A = \cup_k A_k$.  Then, $\pcap(A,B) \leq \sum_k \pcap(A_k, B)$. \\
\indent {\normalfont(iii)} {\normalfont (Positivity)} $\pcap(A,B) > 0$.\\
\indent {\normalfont(iv)} $\lim_{p \to \infty} \pcap(A,B) = 0$ .
\end{lemma}

To estimate $\wcap_q(A,B)$ in (iv) for $q > p$, we will first bound the $\ell^q(V_X)$ norm of $\tau$ and then we will use the expression for the norm comparable to the quantity in inequality (\ref{weak def ineq}).  To be careful in this proof, we will call that quantity  $\norm{\tau}_{q,\infty}^*$ and reserve $\norm{\tau}_{q,\infty}$ for the actual norm.  We will return to mixing the two after proving (iv).

\begin{proof}[Proof of Lemma \ref{basic properties wcap}] Let $Z,X$ and $A,B$ be as in the statement. \\
\indent (i) This follows from the definition as any $\tau$ admissible for $\pcap(A', B')$ is admissible for $\pcap(A,B)$. \\
\indent (ii)  We may assume $\pcap(A_k, B) < \infty$ for all $k$.  Let $\epsilon > 0$.  Let $\tau_k$ be admissible for $\pcap(A_k, B)$ and satisfy $\norm{\tau_k}_{p,\infty}^p < \pcap(A_k, B) + 2^{-k} \epsilon$. Set $\tau = \max(\tau_k)$.  It follows that $\tau$ is admissible for $\pcap(A,B)$.  Now, for $\lambda > 0$ we have
\begin{equation*}
\# \{v : \tau(v) > \lambda\} \leq  \sum_k \# \{v : \tau_k(v) > \lambda\} \leq \sum_k \norm{\tau_k}_{p,\infty}^p / \lambda^p.
\end{equation*}
Hence,
\begin{equation*}
\pcap(A,B) \leq \norm{\tau}_{p,\infty}^p \leq \biggl( \sum_k \pcap(A_k,B) \biggr) + \epsilon.
\end{equation*} \\

(iii) This is proven as \cite[Proposition 4.8]{L1}. \\
\sloppypar
\indent (iv) We assume $\pcap(A,B) < \infty$ here for some $p < \infty$; this is proven for $p = Q$ in Lemma \ref{Qw2 leq ARCdim}.  Let $\tau$ be an admissible vertex function for $\pcap(A,B)$ with $\wpnorm{\tau} < \infty$.  We claim we may assume that there is an $\epsilon > 0$ such that $\tau(v) \leq 1 - \epsilon$ for all $v \in V_X$. To see this, first assume $\tau(v) \leq 1$ for all $v \in V_X$ (we can do this as noted in the introduction).  We see $c = {\sup \{\tau(v) : \tau(v) < 1\} < 1}$ as otherwise there would be infinitely many $v$ with $\tau(v) > 1/2$, contradicting ${\wpnorm{\tau} < \infty}$.  Now, ${T_1 = \#\{v : \tau(v) = 1\}}$ is finite, so there is a level $L > 0$ such that for all $v$ with $\ell(v) \geq L$ we have $\tau(v) \neq 1$.  Set $\epsilon = (1-c)/2$.  We define
\begin{equation*}
\tau'(v) = 
\begin{cases}
1-\epsilon, \text{ if } v \in T_1 \\
\tau(v) + \epsilon, \text{ if } \ell(v) = L \\
\tau(v), \text{ otherwise }
\end{cases}.
\end{equation*}
It follows that $\wpnorm{\tau'} < \infty$ and $\tau'(v) \leq 1 - \epsilon$.  For admissibility of $\tau'$, we see $\tau' \geq \tau$ on $V_X \setminus T_1$, so if $\gamma$ is a path that does not intersect $T_1$, then $\sum_{v \in \gamma} \tau'(v) \geq 1$ follows from the admissibility of $\tau$.  If $\gamma$ intersects $T_1$, then $\gamma$ also intersects $V_X^L$, so by definition ${\sum_{v \in \gamma} \tau'(v) \geq (1-\epsilon) + \epsilon = 1}$.

\fussy
By \cite[equation (3.5.5)]{HKST}, we have
\begin{equation*}
\norm{\tau}_q^q = \int_0^{\infty} q \lambda^{q-1} \#\{v : \tau(v) > \lambda\} d\lambda.
\end{equation*}
Using $\#\{v : \tau(v) > \lambda\} \leq (\wpnorm{\tau}^{*})^p / \lambda^p$ and the fact that $\tau \leq 1-\epsilon$, this is bounded by
\begin{equation*}
q(\wpnorm{\tau}^{*})^p \int_0^{1-\epsilon}\lambda^{q-p-1} d\lambda = \frac{q}{q-p}(\wpnorm{\tau}^{*})^p (1-\epsilon)^{q-p}.
\end{equation*}
It follows that $\norm{\tau}_q^q \to 0$ as $q \to \infty$.  

We now look at the comparable norm in \cite[Equation (6), Section 2]{BnS}.   This is
\begin{equation*}
\wpnorm{\tau} = \sup \{n^{-1+1/q} \sum_{v \in V'} |\tau(v)| : n \in \N, V' \subseteq V_X, \#V' = n\}.
\end{equation*}
If  $\tau \in \ell^q(V_X)$ and $|V'| = n$, then by H\"older's inequality we see that 
\begin{equation*}
\sum_{v \in V'} |\tau(v)| \leq \norm{\tau}_q n^{1/q'}
\end{equation*}
where $1/q + 1/q' = 1$.  Hence, $\norm{\tau}_{q,\infty}^q \leq \norm{\tau}_q^q \to 0$.  As $\wcap_q(A,B) \leq \norm{\tau}_{q,\infty}^q$, the result follows.
\end{proof}

\begin{remark}
In (iv) we can instead bound $(\norm{\cdot}_{q,\infty}^*)^q$ directly in terms of $(\norm{\cdot}_{p,\infty}^*)^p$, but we would still need to compute the comparability constant using the actual norm.  The approach taken also establishes inclusion $\ell^{p,\infty} \subseteq \ell^q$ with norm bounds (if $\tau \leq 1$).
\end{remark}

In \cite{L1}, it is shown that $\wcap$ generalizes modulus in a quasisymmetrically robust way.  Given positively separated open sets $A, B \subseteq Z$, let $\modp(A,B)$ denote the $p$-modulus of the path family connecting $A$ and $B$.  That is, 
\begin{equation*}
\modp(A,B) = \inf \biggl\{ \int_Z \rho^p \biggr\}
\end{equation*}
where the infimum is taken over all Borel functions $\rho : Z \to [0,\infty]$ such that for all rectifiable paths $\gamma$ connecting $A$ to $B$, the integral with respect to path length satisfies $\int_\gamma \rho \ ds \geq 1$.  We make this generalization precise here by recalling the results in the relevant case of positively separated open sets.  Note that the exponent $Q$ in the following is the Ahlfors regular exponent of the space in question.

\begin{thm}[{\cite[Theorem 1.2]{L1}}]
\label{qmod < qcap}
Let $Q>1$ and let $(Z,d,\mu)$ be a compact, connected, Ahlfors $Q$-regular metric space.  Then there exists a constant $C>0$ depending only on $Q$ and the hyperbolic filling parameters with the following property: whenever $A,B \subseteq Z$ are positively separated open sets,
\begin{equation*}
\qmod(A,B)  \leq C\qcap(A,B).
\end{equation*}
\end{thm}

This says that when all exponents match (i.e. $Q$-regular, $\qmod$, and $\qcap$), then $\qcap$ serves as an upper bound for $\qmod$.  For comparability, we use a condition that guarantees $\qmod$ does not vanish (and so our metric space must have many rectifiable curves).  A metric space is called a {\em $Q$-Loewner space} if the function 
\begin{equation*}
\phi_L(t) = \inf \{\qmod(A,B) : \Delta(A,B) \leq t\}
\end{equation*}
has the property that $\phi_L(t) > 0$ for all $t > 0$, where the infimum is taken over pairs of disjoint continua satisfying the given relative distance condition (c.f. \cite{He}).  Loewner spaces were introduced by Heinonen and Koskela in \cite{HK}.

\begin{thm}[{\cite[Theorem 1.3]{L1}}]
\label{qcap < qmod}
Let $Q>1$ and let $(Z,d,\mu)$ be a compact, connected, Ahlfors $Q$-regular metric space which is also a $Q$-Loewner space.  Then there exist constants $c,C>0$ depending only on $Q$ and the hyperbolic filling parameters with the following property: whenever $A,B \subseteq Z$ are positively separated open sets,
\begin{equation*}
c \qmod(A,B) \leq \qcap(A,B) \leq C \qmod(A,B).
\end{equation*}
\end{thm}

In the above statements it was important to look at the weak $Q$-capacity.  When dealing with quasisymmetric maps we may use other exponents.

\begin{thm}[{\cite[Theorem 1.4]{L1}}]
\label{QS inv qcap}
Let $Z$ and $W$ be compact, connected, Ahlfors regular metric spaces and let $p > 1$.  If $f \colon Z \to W$ is an $\eta$-quasisymmetric homeomorphism, then there exist constants ${c,C>0}$ depending only on $\eta$ and the hyperbolic filling parameters with the following property: whenever $A,B \subseteq Z$ are positively separated open sets,
\begin{equation*}
c\pcap(f(A),f(B)) \leq \pcap(A,B) \leq C\pcap(f(A), f(B)).
\end{equation*}
\end{thm}

\subsection{Critical exponents}

Our interest lies in the critical exponents $Q_w$ and $Q_w'$ defined in terms of $\pcap$.  We now work towards showing $Q_w' \leq \ARCdim$.  The proof of this will also complete the proof of (iv) in Lemma \ref{basic properties wcap}.  First we need a lemma that relates relative distances of images of sets under a quasisymmetric homeomorphism.  The statement and proof are similar to \cite[Lemma 3.2]{BnK}.

\begin{lemma}
\label{qs rel dist lemma}
Let $(Z,d)$ and $(Z',d')$ be compact metric spaces.  Let $f \colon Z \to Z'$ be an $\eta$-quasisymmetric homeomorphism.  Then, there is an increasing homeomorphism $\Phi \colon (0,\infty) \to (0,\infty)$ depending only on $\eta$ with $\lim_{t \to \infty} \Phi(t) = \infty$ such that for all positively separated open sets $A,B \subseteq Z$ we have $\Delta(f(A),f(B)) \geq \Phi(\Delta(A,B))$.
\end{lemma}

\begin{proof}
Throughout the proof we use $'$ to denote the images of objects under $f$.  Let $a_n'$ and $b_n'$ be sequences of points in $A'$ and $B'$ such that $\dist(A',B') = \lim_{n \to \infty} d'(a_n',b_n')$.  Let $c_n'$ be a point in $A'$ such that $\diam(A')/3 \leq d'(a_n', c_n')$ and $d_n'$ a point in $B'$ such that $\diam(B')/3 \leq d'(b_n', d_n')$.  Then, as $\eta$ is increasing,
\begin{equation*}
\begin{split}
\eta (\Delta(A,B)^{-1}) &\geq \eta\left(\frac{\min(d(a_n, c_n) , d(b_n, d_n))}{d(a_n, b_n)}\right) \\
&\geq \frac{\min(d'(a_n', c_n'), d'(b_n', d_n'))}{d'(a_n', b_n')} \\
&\geq \frac{\min(\diam(A'), \diam(B'))}{3 d'(a_n',b_n')}.
\end{split}
\end{equation*}
Letting $n \to \infty$ we see $3\eta(\Delta(A,B)^{-1}) \geq \Delta(A',B')^{-1}$.  It follows from this that ${\Delta(A',B') \geq 1 / (3\eta(\Delta(A,B)^{-1}))}$.  Hence we may set $\Phi(t) = 1/(3\eta(t^{-1}))$.
\end{proof}

\begin{remark}\label{rel dist comp}
We may apply the above result to $f^{-1}$ as well to conclude that there exists an increasing function $\Phi' \colon (0,\infty) \to (0,\infty)$ with $\lim_{t \to \infty} \Phi'(t) = \infty$ such that for all positively separated open sets $A,B \subseteq Z$ we have $\Delta(A,B) \geq \Phi'(\Delta(f(A),f(B))$.  Thus,
\begin{equation*}
\Phi(\Delta(A,B)) \leq \Delta(f(A), f(B)) \leq \Phi'^{-1}(\Delta(A,B))
\end{equation*}
\end{remark}

\begin{prop}
\label{Qw2 leq ARCdim}
Let $(Z,d)$ be a compact, connected, Ahlfors regular metric space.  Then, $Q_w'(Z) \leq \ARCdim(Z)$.
\end{prop}

\begin{proof}
Let $p > \ARCdim$.   Let $\theta$ be an Ahlfors regular metric on $Z$ such that $\dim_H(Z,\theta) = p$ (this can be done, for example, by snowflaking).   We work in $(Z,\theta)$, which is justified by Lemma \ref{qs rel dist lemma}.  Let $A$ and $B$ be open with $\dist(A,B) > 0$.  Without loss of generality assume $\diam(A) \leq \diam(B)$ and let $ \Delta = \Delta(A,B) > 0$.  Set $m = \diam(A)$ and $D = \dist(A,B)$, so $\Delta = D / m$. Let $X=(V_X,E_X)$ be a hyperbolic filling for $(Z,\theta)$ with parameter $s>1$.  Consider $g \colon V_X \to \R$ defined by
\begin{equation*}
g(v) = \frac{r(B_v)}{\theta(v,A) \vee D/4} \chi_{\{\theta(v,a) \leq 3D/4\}}
\end{equation*}
where we have used $v$ as both the vertex in $V$ and the center of the corresponding ball $B_v$.

We estimate $\norm{g}_{p,\infty}^p$.  Observe that by Ahlfors regularity, as $\dim_H(Z,\theta) = p$ there is a $C>0$ such that the number of vertices on level $n$ with centers lying in a ball of radius $R$ is bounded above by $C R^p s^{np}$.  For $\lambda > 0$ we have $g(v) > \lambda$ if and only if 
\begin{equation*}
r(B_v) / \lambda > \theta(v,A) \vee D/4 \text{ and } \theta(v,A) \leq 3D/4.
\end{equation*}
Writing $r(B_v) = 2s^{-k}$, this is  
\begin{equation*}
2s^{-k} / \lambda > \theta(v,A) \vee D/4 \text{ and }\theta(v,A) \leq 3D/4.
\end{equation*}
If $2s^{-k} / \lambda > 3D/4$ we see all vertices with $\theta(v,A) \leq 3D/4$ on level $k$ satisfy these inequalities.  From our observation above, there are at most 
\begin{equation*}
C ((3D/4) + m)^p s^{kp}
\end{equation*}
such vertices on level $k$.  If $2s^{-k} / \lambda \leq D/4$ we see no vertices on level $k$ satisfy these inequalities.  For $k$ in between we have $\theta(v,A) < 2s^{-k}/\lambda \leq 3D/4$ and so there are at most 
\begin{equation*}
C ((2 s^{-k} / \lambda) + m)^p s^{kp} = C ((2 / \lambda) + m s^k)^p
\end{equation*}
many vertices satisfying these inequalities.  In this case, $D/4 < 2s^{-k}/\lambda$ so $s^k < 8/D\lambda$.  Thus, we obtain the upper bound
\begin{equation*}
C((2/\lambda) + 8m/D\lambda)^p \lesssim (2 + 8\Delta^{-1})^p / \lambda^p.
\end{equation*}
Combining these estimates, we have
\begin{equation*}
\#\{v:g(v) > \lambda\} \lesssim \sum_{2s^{-k} / \lambda > 3D/4} C ((3D/4) + m)^p s^{kp} + \sum_{D/4 < 2s^{-k} / \lambda \leq 3D/4} (2 + 4\Delta^{-1})^p / \lambda^p.
\end{equation*}
The first sum is geometric and hence estimated by the last term.  This is bounded by a constant times $C ((3D/4) + m)^p (8/3\lambda D)^p$ which is $C (2+ \Delta^{-1} 8/3)^p / \lambda^p$.  The second sum has a number of terms independent of $\lambda$.  Hence, $\norm{g}_{p,\infty}^p \lesssim (1 + \Delta^{-1})^p$.

We now investigate admissibility for $g$.  Consider a path of vertices $\gamma$ with limits in $A$ and $B$.  Consider the subpath $\gamma'$ of $\gamma$ denoted $v_0, v_1, \dots, v_N$ such that $v_0$ is the last vertex with $\theta(v_0, A) \leq D/4$ and $v_N$ is the first vertex after $v_0$ with $\theta(v_N, A) > 3D/4$.  With $B_{v_j} = B(v_j, r_j)$, we have $\theta(v_j, v_{j+1}) < r_j + r_{j+1} < (1 + (1+s))r_j$.  It follows that $\theta(v_{j+1}, A) \leq \theta(v_j, A) + (2+s)r_j$.  Thus,
\begin{equation*}
\sum_{j=0}^{N-1} \frac{(2+s) r_j}{\theta(v_j,A) \vee D/4} \geq \int_{D/4}^{3D/4} \frac{dx}{x}.
\end{equation*}
We now estimate the sum of $g$ over $\gamma$:
\begin{equation*}
\sum_\gamma (2+s) g(v) \geq \sum_{j=0}^{N-1} \frac{(2+s) r_j}{\theta(v_j,A) \vee D/4} \geq \int_{D/4}^{3D/4} \frac{dx}{x} = \log(3).
\end{equation*}

Now, set $\tau =(2+s)g/\log(3)$.  The above shows $\tau$ is admissible for $\pcap(A,B)$.  We have
\begin{equation*}
\norm{\tau}_{p,\infty}^p \lesssim \norm{g}_{p,\infty}^p \lesssim (1 + \Delta^{-1})^p.
\end{equation*}
It follows that $Q_w' \leq p$ and, as $p > \ARCdim$ was arbitrary, we conclude $Q_w' \leq \ARCdim$.
\end{proof}

Recall that $Q_w \leq Q_w'$, so the above also shows $Q_w \leq \ARCdim$.  The following result provides one condition for $Q_w = Q_w' = \ARCdim$.

\begin{lemma}[{\cite[Lemma 5.3]{L1}}]
\label{modulus Q_w bound}
Let $(Z,d,\mu)$ be a compact, connected metric measure space that is Ahlfors $Q$-regular, $Q > 1$, such that there exists $1 \leq p \leq Q$ and a family of paths $\Gamma$ with $\modp(\Gamma) > 0$.  Then there exist positively separated open balls $A$ and $B$ such that for all $q < Q$ we have $\capq(A,B) = \infty$.
\end{lemma}

\section{Path lemmas and control functions}\label{Section control functions}

Recall for $(Z,d,\mu)$ a metric measure space and $X$ a hyperbolic filling, we define $\varphi_p(t) = \sup \{\pcap(A,B) : \Delta(A,B) \geq t\}$.  Unless otherwise stated, in this section it is assumed $Z$ and $X$ are fixed.  We investigate properties of $\varphi_p$ assuming it is not identically $\infty$.  It is clear that $\varphi_p$ is decreasing.

Because we work in a connected metric space and $\pcap$ is monotone, we can compute $\varphi_p(t)$ from only sets which have $\Delta(A,B) = t$

\begin{lemma} 
Let $(Z,d)$ be a compact, connected, Ahlfors regular metric space.  Define $\phi_p(t) = \sup \{\pcap(A,B) : \Delta(A,B) = t\}.$  Then, $\varphi_p = \phi_p$.
\end{lemma}

\begin{proof}
It is clear $\varphi_p \geq \phi_p$.    Suppose there exists $t_0$ such that $\varphi_p(t_0) > \phi_p(t_0)$. This means there are positively separated open sets $A,B \subseteq Z$ with $\Delta(A,B) = t > t_0$ and $\pcap(A,B) > \phi_p(t_0)$.  We may assume $\diam(B) \leq \diam(A)$.  Then, $\dist(A,B) = t \diam(B) > t_0 \diam(B)$.  Set $A' = \{z \in Z : d(z, B) > t_0 \diam(B) \}$.  We see $A'$ is open and $A \subseteq A'$ so $\pcap(A', B) \geq \pcap(A,B)$.  

We show $\dist(A', B) = t_0 \diam(B)$.  By definition, $\dist(A', B) \geq t_0 \diam(B)$.  Suppose $\dist(A', B) = t_1 \diam(B)$ with $t_1 > t_0$.  Set $U = Z \setminus \overline{A'}$, so $U$ is open.  Set $U' = \{z \in Z : d(z, B) > h \diam(B)\}$ where $h = (t_0+t_1)/2$.  Then, $\overline{A'} \subseteq U'$ so $Z = U \cup U'$ forms a disconnection of $Z$, a contradiction.  Hence, $\Delta(A', B) = t_0$ and $\pcap(A', B) \geq \pcap(A,B) > \phi_p(t_0)$ by monotonicity, contradicting the definition of $\phi_p$.
\end{proof}

We now state the main result of this section.

\begin{thm} \label{varphi equivalence}
The following are equivalent: \\
\indent {\normalfont(i)} The function $\varphi_p$ is eventually finite. \\
\indent {\normalfont(ii)} The function $\varphi_p$ is always finite. \\
\indent {\normalfont(iii)} The function $\varphi_p$ has the property that $\lim_{t \to \infty} \varphi(t) = 0$.
\end{thm}

It is clear that (ii) $\implies$ (i) and that (iii) $\implies$ (i). We now show (i) $\implies$ (ii).

\sloppypar
\begin{proof}[Proof of {\normalfont (i)} $\implies$ {\normalfont (ii)}]  Fix $t_0$ such that $\varphi_p(t_0) = C < \infty$.  Let $A, B$ be such that $\Delta(A,B) = t < t_0$.  We may assume $\diam(B) \leq \diam(A)$.  For $b \in B$, consider $B(b, \diam(B))$.  As $Z$ is doubling, we can find at most $N = N(t)$ balls $B_i$ that intersect $B$, have diameter $\diam(B_i) \leq t \diam(B) / (2 + t_0)$, and have the property that $B(b, \diam(B)) \subseteq \cup_i B_i$.  We see $t / (2+t_0) < 1$ so for each $i$ we have $\diam(B_i) < \diam(B)$.  Hence, $\Delta(A, B_i) = \dist(A,B_i) / \diam(B_i)$.  Since $\dist(A, B_i) \geq \dist(A, B) - \diam(B_i)$ it follows that $\Delta(A, B_i) > t_0$.   Thus, $\pcap(A, B_i) \leq \varphi_p(t_0) = C$.  By additivity and monotonicity, $\pcap(A, B) \leq \pcap(A, \cup_i B_i) \leq N(t) C$.
\end{proof}

\subsection{Path lemmas}
\fussy
To show (ii) $\implies$ (iii) is considerably more involved.  To do so, we first develop some path lemmas.  These require binary path structures from \cite{L1} which we explain here for convenience.  

\begin{defn}
Given a vertex $v \in V_X$, a {\em binary path structure} $T_v$ with {\em splitting constant M} is a subgraph with the following properties. \\
\indent (i) Vertices $w \in T_v$ have $\ell(w) \geq \ell(v)$. \\
\indent (ii) There are exactly $2^k$ vertices with level $\ell(v) + kM$.  Denote these $T_v^{kM}$. \\
\indent (iii) For every vertex $w \in T_v^{kM}$, there are exactly two vertices $w_0, w_1 \in T_v^{(k+1)M}$  such that $B_{w_0}, B_{w_1} \subseteq B_w$. \\
\indent (iv) For every vertex $w \in T_v^{kM}$, given $w_0$ and $w_1$ from (iii) there are two paths $\Gamma_w = \{\gamma_{w_0},\gamma_{w_1}\}$ (of vertices and edges), with $\gamma_{w_i}$ connecting $w$ to $w_i$ for each $i = 0,1$, such that the vertices in $\gamma_{w_i}$ strictly increase in level (from $\ell(w)$ to $\ell(w_i)$).  This collection of paths has the property that $T_v = \cup_{k=0}^{\infty} \cup_{w \in T_v^{kM}} \Gamma_w$. \\
\indent (v) Every vertex $x$ in the paths in $\Gamma_w$ from (iv) satisfies $B_x \subseteq 2B_w$.  \\
\indent (vi) If $w, w'$ are vertices with level $\ell(v) + kM$, then $2B_w \cap 2B_{w'} = \emptyset$.  \\
We call $v$ the {\em root} of the binary path structure $T_v$ and say $T_v$ is {\em rooted} at $v$.
\end{defn}

\begin{remark}
\label{2B remark}
We note from (vi) that if $w, w' \in T_v^{kM}$ for some $k$ and $\gamma \in \Gamma_w$ and $\gamma' \in \Gamma_{w'}$, then $\gamma \cap \gamma' = \emptyset$ unless $w = w'$.
\end{remark}

We can construct binary path structures from any point using the following lemma.

\begin{lemma}[{\cite[Lemma 4.4]{L1}}]
\label{Vertex splitting lemma}
There exists a constant $M>0$ depending on the hyperbolic filling parameter $s$ and the Ahlfors regularity constants such that whenever $v \in V_X$ is a vertex in $X$, there exist two vertices $v_1, v_2$ with levels $\ell(v_j) = \ell(v) + M$, $2 B_{v_1} \cap 2 B_{v_2} = \emptyset$, and $B_{v_1}, B_{v_2} \subseteq B_v$.
\end{lemma}

We do not prove this here, but note that it follows from an Ahlfors regularity computation.  

Our first path lemma is a refinement of \cite[Lemma 4.7]{L1}.  It involves {\em ascending edge paths} in a binary path structure.  These are paths such that the corresponding vertex paths start at the root $v$ and strictly increase in level.  This means if $\{v_k\}$ is the sequence of vertices our path travels through, then we must have $\ell(v_{k+1}) = \ell(v_k) + 1$ for all $k$.  We call $\tau$ {\em admissible} for ascending edge paths if the $\tau$-length of each ascending edge path is at least $1$, where the $\tau$-length of an edge path $\gamma$ is $\sum_{e \in \gamma} \tau(e)$.

\begin{lemma}
\label{First Path Lemma}
Let $v \in V_X$ and let $T_v$ be a binary path structure with splitting constant $M$ and root vertex $v$.  Let $p > 1$.  Then, there is a function $S < \infty$ depending on $p$ and $M$ with the following property: whenever $\tau \colon E_X \to [0,\beta]$ is a function with $\norm{\tau}_{p,\infty} \leq a$, then there is an ascending edge path in $T_v$ with $\tau$-length bounded above by $S(a,\beta)$.  Moreover, $S(a,\beta) \to 0$ as $\beta \to 0$ for fixed $a$.
\end{lemma}

The proof is similar to that of Lemma 4.4 in \cite{L1}; we include it here for completeness.  We use the notion of the level of an edge $e$ which we define as the minimum level of the vertices incident to $e$.

\begin{proof}
Suppose $T_v$ is as above and $\tau : E_X \to [0,\beta]$ is a function with $\norm{\tau}_{p,\infty} \leq a$.  In this proof, we will refer to the level of a vertex to mean the distance that vertex is from $v$.  That is, the level of a vertex $w$ in $T_v$ is $\ell(w) - \ell(v)$ where $\ell$ is the level function from $X$.  Note that there are at least $M$ edges in the first $M$ levels of $T_v$, at least $2M$ edges in levels $M+1$ to $2M$ of $T_v$, and in general at least $2^{k-1}M$ edges in the levels $(k-1)M+1$ to $kM$ in $T_v$.  Hence, the total number of edges up to level $kM$ is at least
\begin{equation*}
\sum_{j=1}^k 2^{j-1}M = (2^{k}-1)M.
\end{equation*}
We bound the total path length $\tau$ gives to paths.  That is, if $\Gamma_k$ is the collection of ascending paths starting at $v$ and ending at level $kM$, then we bound
\begin{equation*}
\sum_{\gamma \in \Gamma_k} \sum_{e \in \gamma} \tau(e).
\end{equation*}
This is increased the more mass is placed on lower level edges, which we use to bound this quantity.  From the weak norm definition, we have
\begin{equation}
\label{weak lp}
\#\{e : \tau(e) > \lambda\} \leq a^p / \lambda^p.
\end{equation}
Choose $N \in \N_0 = \N \cup \{0\}$ such that $(2^N - 1)M < a^p/\beta^p \leq (2^{N+1}-1)M$.  We overestimate by assuming all edges up to level $(N+1)M$ carry weight $\beta$.  Inequality (\ref{weak lp}) then tells us all other weights satisfy $\tau(e) \leq a((2^{N+1}-1)M)^{-1/p}$.  Using this bound for the next (at least) $2^{N+1}M$ edges up to level $(N+2)M$, we see the weights placed on edges with levels more than $(N+2)M$ satisfy $\tau(e) \leq a((2^{N+2}-1)M)^{-1/p}$.  Continuing in this manner, we obtain the following bound:
\begin{equation*}
\sum_{\gamma \in \Gamma_k} \sum_{e \in \gamma} \tau(e) \leq 2^k \biggl(M \beta (N+1) + \sum_{j=N+1}^{k-1} \frac{M a}{[(2^{j} - 1)M]^{1/p}} \biggr).
\end{equation*}
There are at least $2^k$ different ascending edge paths from $v$ to level $kM$.  It follows that for each $k$, there is at least one path $\gamma_k$ from $v$ to level $kM$ with $\tau$-length bounded above by 
\begin{equation*}
\sum_{e \in \gamma_k} \tau(e) \leq \biggl(M \beta (N+1) + \sum_{j=N+1}^{k-1} \frac{M a}{[(2^{j} - 1)M]^{1/p}} \biggr).
\end{equation*}
We run a diagonalization argument to find an infinite ascending edge path.  This goes as follows: first, we can pass to a subsequence and assume that the paths $\gamma_k$ agree on the first $M$ edges.  Then, we can pass to a further subsequence and assume that the paths $\gamma_k$ agree on the first $2M$ edges.  Repeating this we construct a path $\gamma$ with $\tau$-length bounded above by
\begin{equation*}
S(a,\beta) = \biggl(M \beta (N+1) + \sum_{j=N+1}^\infty \frac{M a}{[(2^{j} - 1)M]^{1/p}} \biggr).
\end{equation*}

We compute from $(2^{N} - 1)M < a^p / \beta^p$ that $N < \log(1 + \frac{a^p}{ \beta^p M}) / \log(2)$.  Hence, we see $M\beta (N+1) \to 0$ as $\beta \to 0$ for fixed $a$.  Likewise, from our other estimate, $N \geq \log(1 + \frac{a^p}{\beta^p M})/\log(2) - 1$ and so $N \to \infty$ as $\beta \to 0$ for fixed $a$.  Thus, the second term in the expression for $S(a,\beta)$ tends to $0$ as $\beta \to 0$ as this is a convergent series.  Hence, for fixed $a$, we see $S(a,\beta) \to 0$ as $\beta \to 0$.  
\end{proof}

In what follows we will need a stronger conclusion.  Above we bound the average path length from $\tau$ by $S(a,\beta)$.  Instead, we wish to guarantee at least some number $m$ paths have a ``small'' bounded length.  For this, we first prove the result abstractly and then use 
\begin{equation*}
S_k(a,\beta) = \biggl(M \beta (N+1) + \sum_{j=N+1}^{k-1} \frac{M a}{[(2^{j} - 1)M]^{1/p}} \biggr),
\end{equation*} 
the average bound we found above up to to level $k$ with $k > N+2$.  
 
\begin{lemma}\label{Path average lemma}
Suppose $\Gamma$ is a finite collection of paths and $\tau \colon E_X \to [0,\infty)$ is a function such that the average $\tau$-length of paths in $\Gamma$ is $\alpha$.  Let $h > \alpha$.  Then, there are at least $|\Gamma|(1-\alpha / h)$ paths with $\tau$-length bounded above by $h$.
\end{lemma}

\begin{proof}
The total $\tau$-length given to paths is $|\Gamma|\alpha$.  Suppose $m$ paths have $\tau$-length greater than $h$.  Then, we must have $mh \leq |\Gamma|\alpha$ and so $m \leq |\Gamma|\alpha / h$.  Thus, there are at least $|\Gamma| - |\Gamma|\alpha / h = |\Gamma|(1-\alpha / h)$ paths with $\tau$-length bounded above by $h$.
\end{proof}

In our situation we use the paths formed by concatenating those from $\Gamma_w$ in the definition of a binary path structure, so $|\Gamma| = 2^k$.  We also have $\alpha \leq S_k(a,\beta) \leq S(a,\beta)$.  Hence, there are at least
\begin{equation*}
2^k (1 - S(a,\beta)/h)
\end{equation*}
paths with $\tau$-length bounded above by $h$.

\begin{lemma}[Main Path Lemma]
\label{Last Path Lemma}
Let $T_v$ be a binary path structure with splitting constant $M$ and let $p>1$.  Suppose $\tau : E_X \to [0,1]$ satisfies $\norm{\tau}_{p,\infty} \leq a$.  Let $\delta > 0$.  Then, there is a real number $\beta > 0$ and level $\ell_0 = \ell_0(a, \beta, M, p)$ with the following property: if $\tau(e) \leq \beta$ for all $e$ with $\ell(e) \leq \ell_0$ (where $\ell(e)$ is the maximal level of the endpoints of $e$), then there is an infinite ascending path $\gamma$ with $\sum_{e \in \gamma} \tau(e) < \delta$.
\end{lemma}

Here we use the level $\ell$ to mean distance from $v$, as before.  Intuitively the $\beta$ bound forces the initial portions of ascending paths to have small $\tau$-length while the bound on $\norm{\tau}_{p,\infty}$ together with the exponential growth of $T_v$ forces there to be many paths with ending portions with small $\tau$-length.

\begin{proof}[Proof of Lemma \ref{Last Path Lemma}]
Let $T_v$, $M$, $p$, and $\tau$ be as above.  Let $\beta$ be such that $3S(a,\beta) < \delta$.   Let $h = 2S(a,\beta)$ in Lemma \ref{Path average lemma} and choose $k$ such that $2^{k-1} > a^p / \beta^p + 1$.  If $\tau \leq \beta$ up to level $Mk$ then, from property (iv) of binary tree structures, there are $2^k$ ascending paths beginning at $v$ and ending at level $kM$ with the property that if $w, w'$ are distinct ending vertices, then $2B_{w} \cap 2B_{w'} = \emptyset$.  These paths have average path length bounded above by $S(a,\beta)$.  Associate to each of these paths the subtree of ascending paths beginning at the highest level vertex.  It follows that these subtrees are disjoint.  By Lemma \ref{Path average lemma}, at least $2^k (1-S(a,\beta)/2S(a,\beta)) = 2^{k - 1}$ of these paths have length bounded above by $2S(a,\beta)$.  As $|\{e : \tau(e) > \beta\}| \leq a^p / \beta^p$ it follows that the restriction of $\tau$ to at least one of these subtrees is bounded by $\beta$.  On this subtree we apply Lemma \ref{First Path Lemma} to conclude there is an ascending path with $\tau$-length bounded above by $S(a,\beta)$.  Concatenate these paths and set $\ell_0 = kM$. 
\end{proof}

\begin{remark}
In Lemma \ref{Last Path Lemma} it follows that, for fixed $a$, as $\delta \to 0$ we have $\beta \to 0$ and $\ell_0 \to \infty$
\end{remark}

\begin{remark}\label{vertex remark}
Lemmas \ref{First Path Lemma} and \ref{Last Path Lemma} (up to a multiplicative factor) hold for ascending vertex paths by Lemma \ref{vertex edge comp}.
\end{remark}


\subsection{Structures in the hyperbolic filling}
With binary path structures and path lemmas in hand, we continue to work towards the implication (ii) $\implies$ (iii) in Theorem \ref{varphi equivalence}.  To prove the result we will use the notion of the hull of a set in a hyperbolic filling $X$.   Recall we use the variable $s > 1$ for the parameter of the hyperbolic filling $X$.

\begin{defn}\label{hull def}
Let $A \subseteq Z$.  Define $j_A \in \N$ by 
\begin{equation*}
-\log_s(\diam(A)) < j_A \leq -\log_s(\diam(A)) + 1
\end{equation*}
so $\diam(A)/s \leq s^{-j_A} < \diam(A)$.  The {\em hull} of $A$ in $X$ is defined as 
\begin{equation*}
H(A) = H_A = \{v \in V_X : \ell(v) \geq j_A \text{ and } B_v \cap A \neq \emptyset\}.
\end{equation*}
\end{defn}

\begin{lemma}\label{ring thickness lemma}
Let $(Z,d)$ be a compact, connected, Ahlfors regular metric space and $X$ a hyperbolic filling of $Z$.  Let $K > 0$.  There are constants $c_0(K) > 0$ and $R_{max}(Z,X)$ such that if $0 < r_1 < r_2 < R_{max}$ satisfy $r_2 / r_1 > c_0$, then for any $z \in Z$ we have 
\begin{equation*}
\rho(H(B(z, r_1)), X \setminus H(B(z, r_2))) \geq K.
\end{equation*}
where $\rho$ denotes distance in $X$. 
\end{lemma}

\sloppypar
\begin{proof}
We choose $R_{max}$ such that for any $z \in Z$ the set $X \setminus H(B(z,r_2))$ is nonempty. Fix $z \in Z$ and let $H_i = H(B(z,r_i))$ for $i \in \{1,2\}$, so we wish to show ${\rho(H_1, X \setminus H_2) \geq K}$.  Consider a chain of neighboring vertices $v_0, v_1, v_2, \dots, v_m$ such that $v_0 \in H_1$ and $v_m \notin H_2$.  From the definition of the hulls, we have integers $j_i = j_{B(z,r_i)}$.  There are two cases, either (i) $\ell(v_m) < j_2$ or (ii) $\ell(v_m) \geq j_2$ and $B_{v_m} \cap B(z,r_2) = \emptyset$.

\fussy
Assume (i).  We have $\ell(v_0) \geq j_1$ and so $m \geq j_1 - j_2$.  We have 
\begin{equation*}
j_1 - j_2 \gtrsim \log_s(\diam(B(z,r_2)) / \diam(B(z,r_1)) \gtrsim \log_s(c r_2 / r_1)
\end{equation*}
for a constant $c>0$ as Ahlfors regular spaces are uniformly perfect.  Thus, if $r_2 / r_1$ is large then so is $m$.

Now, assume (ii).  View the vertices $v_k$ as the centers of their respective balls in $Z$.  Then, by the triangle inequality we have
\begin{equation*}
d(v_0, v_m) \geq d(z,v_m) - d(z, v_0) \geq r_2 - r_1 - r(B_{v_0}).
\end{equation*}
We see 
\begin{equation*}
r(B_{v_0}) \leq 2s^{-j_1} < 2 \diam(B(z,r_1)) \leq 4 r_1
\end{equation*}
and so 
\begin{equation*}
d(v_0, v_m) \geq r_2 - 5r_1.
\end{equation*}
We also have
\begin{equation*}
d(v_k, v_{k+1}) \leq 2s^{-\ell(v_k)} + 2s^{-\ell(v_{k}) + 1}.
\end{equation*}
Hence,
\begin{equation*}
d(v_0, v_m) \leq \sum_{k=1}^{m} 4s^{-j_1 + k} = 4s^{-j_1} \frac{s(s^m - 1)}{s-1} \lesssim r_1 s^{m+1}.
\end{equation*}
Combining these two inequalities, we have
\begin{equation*}
\frac{r_2}{r_1} - 5 \lesssim s^m
\end{equation*}
and so if $r_2 / r_1$ is large then so is $m$.
\end{proof}

We record the following observation that paths connecting the separated regions $H_1$ and $X \setminus H_2$ as in Lemma \ref{ring thickness lemma} must at some point be far from both sets.

\begin{lemma}\label{ring sep lemma}
Let $(Z,d)$ be a compact, connected, Ahlfors regular metric space and $X$ a hyperbolic filling of $Z$.  Suppose $0<r_1<r_2<R_{max}$ and let $z \in Z$.  Set $H_i = H_{B(z,r_i)}$.  Let $L \in \N$ and suppose $\rho(H_1, X \setminus H_2) \geq 2L+1$.  Define $R = H_2 \setminus H_1$.  Let $\gamma$ be an edge path in $X$ such that there are edges $e_1,e_2 \in \gamma$ with $e_1 \in H_1$ and $e_2 \notin H_2$.  Then, there is an edge $e \in E_X$ with $e \in H_2 \setminus H_1$ such that $\rho(e,H_1) \geq L$ and $\rho(e, X \setminus H_2) \geq L$.
\end{lemma}

\begin{proof}
This follows immediately as the vertices of a given edge are at distance $1$ from one another; tracing $\gamma$ from $e_1$ to $e_2$ we must encounter such an edge.
\end{proof}

If $\gamma, H_1, H_2$ and $R$ are defined as above, we say $\gamma$ {\em travels through} $R$.

\begin{lemma}\label{NK lemma}
Let $(Z,d)$ be a compact, connected, Ahlfors regular metric space and $X$ a hyperbolic filling of $Z$.  Let $N, K \in \N$ and $c_1 > 0$.  Then, there exists $c_2 = c_2(K,N)$ such that if $A, B \subseteq Z$ satisfy $\diam(B) \leq \diam(A)$ and $\Delta(A,B) \geq c_2$ then, given $b \in B$, we can construct a set of $N+1$ concentric balls $B_k = B(b,r_k)$ with $r_0 = \diam(B)$ such that $A \cap B_N = \emptyset$, the radii satisfy $r_{k+1} / r_k \geq c_0$, and for each $0 \leq k < N$ we have $\Delta(10 B_k, Z \setminus B_{k+1}) \geq c_1$.
\end{lemma}

Here $c_0 = c_0(K)$ is from Lemma \ref{ring thickness lemma}.

\sloppypar
\begin{proof}
In what follows we note we may assume $r_{k+1}$ is small enough such that ${Z \setminus B(b,r_{k+1}) \neq \emptyset}$. We see that if $r_{k+1} > 10 r_k$ then
\begin{equation*} 
\Delta(10B(b,r_k), Z \setminus B(b,r_{k+1})) \geq \frac{r_{k+1} - 10 r_k}{20r_k} = \frac{r_{k+1}}{20 r_k} - \frac{1}{2}.
\end{equation*}
Thus, to satisfy the relative distance condition and the $c_0(K)$ condition it suffices to choose the successive ratios $r_{k+1}/r_k$ to be large enough, say $r_{k+1}/r_k \geq c_3(K)$.  We choose $r_0 = \diam(B)$ and $r_{k+1} = c_3 r_{k}$.  As $\diam(Z) \leq 1$ and $\diam(B) \leq \diam(A)$, we see $\diam(B) \leq 1/\Delta(A,B)$.  Thus, for fixed $N$, there is a constant $c(N)$ such that if $c_2 > c(N)$, then $r_N < R_{max}$ (with $R_{max}$ from Lemma \ref{ring thickness lemma}).  If $B_N \cap A \neq \emptyset$ then $\dist(A,B) \leq r_N = c_3^N \diam(B)$ and so 
\begin{equation*}
\Delta(A,B) = \dist(A,B) / \diam(B) \leq c_3^N.
\end{equation*}
Thus $c_2 = (c_3^N \vee c(N)) + 1$ satisfies the claim.
\end{proof}

\fussy
\begin{lemma} \label{tree lemma}
Let $(Z,d)$ be a compact, connected, Ahlfors regular metric space and $X$ a hyperbolic filling of $Z$.  Write $X = (V_X, E_X)$ and let $s$ be the filling parameter of $X$.

{\normalfont (i)} Let $z \in Z, r>0$, and set $B = B(z,r)$.  Suppose $v \in V_X$ satisfies $v \in H(B)$.  Then, there is a binary path structure $T_v$ with splitting constant $M$ depending only on $(Z,d,\mu)$ and $s$, originating at $v$, such that all ascending paths have boundary limits in $10B$.  

{\normalfont (ii)}  There exists $R_{max}' > 0$ with the following property.  Let $z \in Z$ and let $B = B(z,r)$ and $B' = B(z,t)$ with $0 < r < t < R_{max}'$.  Suppose $v \in V_X$ satisfies $v \in H(B') \setminus H(B)$.  Then there is a binary path structure $T_v$ with splitting constant $M'$ depending only on $(Z,d,\mu)$ and $s$, originating at $v$, such that all ascending paths (in this case geodesic paths in $T_v$ originating at $v$) have boundary limits in $Z \setminus \overline{B}$.
\end{lemma}

\begin{proof}
We first show (i).  As $v \in H(B)$ we see $B_v \cap B \neq \emptyset$ and $\diam(B_v) \leq 8r$.  Thus, $B_v \subseteq 10B$.  The result then follows from repeatedly applying Lemma \ref{Vertex splitting lemma} to $B_v$.

We now show (ii).  As $v \notin H(B)$, either $B_v \cap B = \emptyset$ or $B_v \cap B \neq \emptyset$ and $\ell(v) < j_B$.  

If $B_v \cap B = \emptyset$ we form our structure as in (i).  To guarantee the limits lie in $Z \setminus \overline{B}$ we first pass to a vertex $w \in V_X$ with $\ell(w) = \ell(v) + 1$ and $v \in B_w$.  This guarantees $\overline{B_w} \cap \overline{B} = \emptyset$, so we may apply this construction. 

Suppose $B_v \cap B \neq \emptyset$ and $\ell(v) < j_B$.  Consider a parent $x$ of $v$ on level $\ell(v) - j$, where we define $R_{max}'$ such that $\ell(v) - j > 0$.  The maximal $j$ required will only depend on $(Z,d,\mu)$ and $s$ below, so $R_{max}'$ can be defined this way.  We have $r(B_x) = r(B_v) s^{j}$. Let $c_Q$ and $C_Q$ denote the Ahlfors regularity constants of $Z$, so for every ball $B_r$ of radius $r < \diam(Z)$ we have $c_Q r^Q \leq \mu(B_r) \leq C_Q r^Q$.  Then, 
\begin{equation*}
\mu(B_x \setminus (4+s) B_v) \geq c_Q r(B_v)^Q s^{jQ} - C_Q ((4+s)r(B_v))^Q
\end{equation*}
which is positive for large enough $j$ depending only on $(Z,d,\mu)$ and $s$.  Now, choose a point
$z' \in B_x \setminus (4+s) B_v$.  There is a vertex $w$ with level $\ell(w) = \ell(v)$ such that $z' \in B_w$.  We claim $\overline{B_w} \cap \overline{B} = \emptyset$.  As $\ell(v) < j_B$, we have 
\begin{equation*}
2\diam(B)/s \leq 2s^{-j_B} \leq 2s^{-\ell(v)} = r(B_v).
\end{equation*}
Thus, as every point $b \in B$ satisfies $d(b,v) \leq r(B_v)+\diam(B)$, we have $\overline{B} \subseteq (1+s)B_v$.  We see $d(v,z') \geq (4+s) r(B_v)$ and so if $w' \in B_w$ we have
\begin{equation*}
\begin{split}
d(w', v) &\geq d(v, z') - d(z', w') \geq (4+s) r(B_v) - 2 r(B_v) = (2+s)r(B_v).
\end{split}
\end{equation*}
Hence, $\dist(w', (1+s)B_v) \geq r(B_v)$.  Thus, $\overline{B_w} \cap \overline{B} = \emptyset$.  The result now follows from applying Lemma \ref{Vertex splitting lemma} to $B_w$.  The key here is that there is a controlled number of steps to get from $B_v$ to $B_w$ going through $B_x$ which impacts the resulting $M'$.
\end{proof}

\begin{remark}\label{tree remark}
By choosing the maximum of the values $M$ and $M'$ above, we may assume there is a uniform $M$ that works in both situations.
\end{remark}

\subsection{Proof of the last implication in Theorem \ref{varphi equivalence}}

\begin{proof}[Proof of {\normalfont (ii) $\implies$ (iii)} in Theorem \ref{varphi equivalence}]
Let $\epsilon > 0$ and $p > 1$.  We will show if $\Delta(A,B)$ is large enough, then $\pcap(A,B) < \epsilon$.  We assume $\diam(B) \leq \diam(A)$ and $\varphi_p(1) < \infty$.  Choose $\beta$ from Lemma \ref{Last Path Lemma} with $\delta = 1/4$, $a = \varphi_p(1)$ and $M$ from Remark \ref{tree remark}.  Let $K = 2\ell_0(\beta) + 3$ where $\ell_0(\beta)$ is from Lemma \ref{Last Path Lemma}.  Let $n_1, n_2 \in \N$ be such that $2^p \varphi_p(1) / n_1^{p-1} < \epsilon$ and $n_2 > 2((\varphi_p(1)^2 / \beta^2) + 2)$.  Let $N = n_1 (2 n_2 + 2)$.

Let $\Delta(A,B) > c_2$ from Lemma \ref{NK lemma} with $N, K$ defined above and $c_1 = 1$.  Fix $b \in B$ and construct the balls $B_k$ as in that lemma.  From this we obtain $N$ rings $H(B_{k+1}) \setminus H(B_k)$ in $X$ for $k \in \{0,\dots, N-1\}$; we partition these in increasing order into $n_1$ different packages $P_i$ of $2n_2 + 2$ rings, say package $i$ contains the rings $R_j = R_j(i) = H(B_{q_i+j}) \setminus H(B_{q_i+j-1})$ for $q_i = i(2n_2 + 2)$ and $j \in \{1,\dots,2n_2 + 2\}$.  Consider the ring $R_{q_i+n_2 + 1}$.  By construction, 
\begin{equation*}
\Delta(10 B_{q_i+n_2}, Z \setminus \overline{B_{q_i+n_2+1}}) \geq 1
\end{equation*}
and so there is an admissible $\tau$ for $\pcap(10 B_{q_i+n_2}, Z \setminus \overline{B_{q_i+n_2+1}})$ such that $\norm{\tau}_{p,\infty}^p \leq \varphi_p(1)$.  We claim $2\tau_i = 2\tau|_{P_i}$ is admissible for $\pcap(10 B_{q_i+n_2}, Z \setminus \overline{B_{q_i+n_2+1}})$.

Let $\gamma$ be an edge path in $X$ connecting $10 B_{q_i+n_2}$ and $Z \setminus \overline{B_{q_i+n_2+1}}$.  As the rings $R_j$ are disjoint, it follows from the definition of $n_2$ that there are $k \in \{1,\dots,n_2\}$ and $k' \in \{n_2 + 2, \dots, 2n_2 + 2\}$ such that for all $e \in R_{q_i+k}, R_{q_i+k'}$ we have $\tau(e) \leq \beta$.  Assume the path $\gamma$ travels through at least one of  $R_{q_i + k}$ and $R_{q_i + k'}$; otherwise it only lies in $P_i$ and so $2\tau_i$ is admissible for $\gamma$.  From any vertex in $R_{q_i + k}$ or $R_{q_i + k'}$ at distance at least $(K-1)/2$ from the boundaries of these sets we may form a binary path structure as in Lemma \ref{tree lemma} with limits in $10 B_{q_i+n_2}$ for $R_{q_i + k}$ and in  $Z \setminus \overline{B_{q_i+n_2+1}}$ for $R_{q_i + k'}$.  Such a path structure avoids $\beta$ for long enough to satisfy the conditions in Lemma \ref{Last Path Lemma} and hence contains an ascending edge path of $\tau$-length $< 1/4$.  Thus, the portion of $\gamma$ in $P_i$ must have $\tau$-length $\geq 1/2$.  Hence, $2\tau_i$ is admissible.  As $A \subseteq Z \setminus \overline{B_{q+n_2+1}}$ and $B \subseteq 10 B_{q+n_2}$, it follows that $2\tau_i$ is admissible for $\pcap(A,B)$ as well.  Note that the functions $2\tau_i$ have disjoint support.

Now, set 
\begin{equation*}
T = \frac{1}{n_1} \sum_{i=1}^{n_1} 2\tau_i.
\end{equation*}
As each $2\tau_i$ is admissible for $\pcap(A,B)$ it follows that $T$ is as well.  We estimate $\norm{T}_{p,\infty}$ recalling that the functions $\tau_i$ have disjoint support:
\begin{equation*}
\begin{split}
\#|\{e : T(e) > \lambda\}| &= \#|\{e : \exists i \text{ s.t. } 2\tau_i(e)/{n_1} > \lambda\}| \\
&= \sum_{i=1}^{n_1} \#|\{e : \tau_i(e) > \lambda n_1 / 2\}| \\
&\leq n_1 2^p \varphi_p(1) / (\lambda n_1)^p \\
&= (2^p \varphi_p(1) /n_1^{p-1}) (1/\lambda^p)
\end{split}
\end{equation*}
which, by our choice of $n_1$, is bounded by $\epsilon / \lambda^p$.
\end{proof}

\begin{remark}\label{T support}
The admissible function $T$ is supported on $H(B_N)$.  This is the hull $H(B(a, c_3^N \diam(A)))$, where $c_3$ depended on $K$.
\end{remark}

\section{Quasisymmetric uniformization}\label{Section qs unif}

In this section we prove Theorem \ref{qs unif thm}, a result analogous to \cite[Theorem 10.4]{BnK}, as well as its corollaries.

The converse direction follows from Theorem \ref{QS inv qcap} and Remark \ref{rel dist comp} as $\Stwo$ is a Loewner space; alternatively this follows from Lemma \ref{Qw2 leq ARCdim} as the conditions guarantee that $Q_w' \leq 2$.  The stronger conditions on $\varphi_2$ follow from Theorem \ref{varphi equivalence}.

Comparing this to \cite[Theorem 10.4]{BnK} we have the following observations.  The control function $\Psi(t)$ used in \cite[Theorem 10.4]{BnK} was required to tend to $0$ as $t \to \infty$.  In our setting $\varphi_2$ plays the analogous role.  Here the limiting behavior is not required a priori by Theorem \ref{varphi equivalence}.  The statement of \cite[Theorem 10.4]{BnK} is also fairly technical.  While the statement of Theorem \ref{qs unif thm} has technicalities as well, particularly in the definition of the hyperbolic filling, it has the advantage of being simpler to parse.

To prove Theorem \ref{qs unif thm}, we will use a number of results from \cite{BnK}.  For clarity and convenience, we list the full statements with some discussion in the next few pages.  We omit proofs.

\subsection*{Results from \cite{BnK} and preliminaries} 

It will be easier to work with a quasi-M\"obius condition instead of the quasisymmetric condition in what follows.  In a metric space $(Z,d)$, the {\em cross-ratio} of four distinct points $z_1, z_2, z_3, z_4 \in (Z,d)$ is given by
\begin{equation*}
[z_1, z_2, z_3, z_4] = \frac{d(z_1, z_3) d(z_2, z_4)}{d(z_1, z_4) d(z_2, z_3)}.
\end{equation*}

\begin{defn}
\begin{sloppypar}
Given an increasing function $\eta \colon (0,\infty) \to (0,\infty)$ with $\lim_{t \to 0} \eta(t) = 0$, a homeomorphism ${f \colon (Z,d) \to (Z', d')}$ is {\em $\eta$-quasi-M\"obius} if for every set of four distinct points $z_1, z_2, z_3, z_4 \in (Z,d)$ with images $f(z_i) = z_i' \in (Z', d')$ we have
\end{sloppypar}
\begin{equation*}
[z_1', z_2', z_3', z_4'] \leq \eta([z_1, z_2, z_3, z_4]).
\end{equation*}
\end{defn}

A quantity related to the cross-ratio, called the {\em modified cross-ratio}, is given by
\begin{equation*}
\langle z_1, z_2, z_3, z_4 \rangle = \frac{d(z_1, z_3) \wedge d(z_2, z_4)} {d(z_1, z_4) \wedge d(z_2, z_3)}.
\end{equation*}

By \cite[Lemma 2.3]{BnK}, the function $\eta_0(t) = 3(t \vee \sqrt{t})$ has the property that for every set of four distinct points $z_1, z_2, z_3, z_4 \in (Z,d)$ we have 
\begin{equation*}
\langle z_1, z_2, z_3, z_4 \rangle \leq \eta_0([z_1, z_2, z_3, z_4]).
\end{equation*}

\begin{remark}
By permuting $z_1, z_2, z_3, z_4$ there is similar control of the modified cross-ratio by the cross-ratio from below.
\end{remark}

On bounded spaces quasi-M\"obius maps are quasisymmetric maps (cf. \cite{Va}).  Thus, the limiting map attained from \cite[Lemma 3.1]{BnK} at the end of the proof will be our desired quasisymmetric map.

The basic strategy of the proof of Theorem \ref{qs unif thm} follows that of \cite{BnK}.  We briefly state how each of the lemmas here will be used in the proof that follows.  We leave the technical definitions to \cite{BnK}.

First, we will set up denser and denser graphs in $Z$ and $\Stwo$ which are triangulations by using the following proposition.

\begin{prop}[\cite{BnK}, Proposition 6.7]
Suppose $(Z,d)$ is a metric space homeomorphic to $\Stwo$.  If $(Z,d)$ is $C_0$-doubling and $\lambda$-LLC, then for given $0<r<\diam(Z)$ and any maximal $r$-separated set $A \subseteq  Z$ there exists an embedded graph $G = (V,E)$ which is the $1$-skeleton of a triangulation $T$ of $Z$ such that:\\
\indent \text{\textnormal{(i)}} The valence of $G$ is bounded by $K$.\\
\indent \text{\textnormal{(ii)}} The vertex set $V$ of $G$ contains $A$.\\
\indent \text{\textnormal{(iii)}} If $e \in E$ then $\diam(e) < K r$.  If $u,v \in V$ and $d(u,v) < 2r$ then $k_G(u,v) < K$.\\
\indent \text{\textnormal{(iv)}} For all balls $B(a,r) \subseteq Z$ we have $\# (B(a,r) \cap V) \leq K$.\\
The constant $K \geq 1$ depends only on $C_0$ and $\lambda$.
\end{prop}

Here $k_G(u,v)$ is the combinatorial distance between $u,v$ in the graph $G$.  Property (iii) will be used to show that we can combine these triangulations into a structure $Y$ which is quasi-isometric to our hyperbolic filling.  To use other results in \cite{BnK}, we need to know that these structures form $K$-approximations .  This is done in $Z$ by associating the objects $p(y) = y$, $r(y) = s^{-n}$, and $U_y = B(p(y), Kr(y))$ to form $p, r$, and $\mathscr{U}$ (with notation from \cite{BnK}).

\begin{cor}[\cite{BnK}, Corollary 6.8]
$(G, p, r, \mathscr{U})$ is a $K'$-approximation of $Z$, where $K'$ depends only on $\lambda$ and $C_0$.
\end{cor}

In $\Stwo$ we use properties of normalized circle packings.  For this, a normalized circle packing $\mathscr{C}$ of $\Stwo$ realizing a triangulation $G$ has objects $p^\mathscr{C}, r^\mathscr{C}, \mathscr{U}^\mathscr{C}$ associated to it; see \cite[Section 5]{BnK}.

\begin{lemma}[\cite{BnK}, Lemma 5.1]
Suppose $G$ is combinatorially equivalent to a $1$-skeleton of a triangulation of $\Stwo$, and $\mathscr{C}$ is a normalized circle packing realizing $G$.  Then $(G, p^\mathscr{C}, r^\mathscr{C}, \mathscr{U}^\mathscr{C})$ is a $K$-approximation of $\Stwo$ with $K$ depending only on the valence of $G$.
\end{lemma}

We use these structures to define coarse quasi-M\"obius maps between vertex subsets.  This is done with the aid of the following Lemma.

\begin{lemma}[\cite{BnK}, Lemma 4.7]
\begin{sloppypar}
Suppose $(Z,d)$ is a connected metric space and ${((V,\sim), p, r, \mathscr{U})}$ is a $K$-approximation of $Z$.  Suppose $L \geq K$ and $W \subseteq V$ is a maximal set of combinatorially $L$-separated vertices.  Then $M = p(W) \subseteq Z$ is weakly $\lambda$-uniformly perfect with $\lambda$ depending only on $L$ and $K$.
\end{sloppypar}
\end{lemma}

We show a subsequence of these maps is actually uniformly quasi-M\"obius.  If not, then \cite[Lemma 3.3]{BnK} would fail.

\begin{lemma}[\cite{BnK}, Lemma 3.3]
Suppose $(X, d_X)$ and $(Y, d_Y)$ are metric spaces, and $f \colon X \to Y$ is a bijection.  Suppose that $X$ is weakly $\lambda$-uniformly perfect, $Y$ is $C_0$-doubling, and there exists a function $\delta_0 \colon (0,\infty) \to (0,\infty)$ such that
\begin{equation*}
[f(x_1), f(x_2), f(x_3), f(x_4)] < \delta_0(\epsilon) \implies [x_1, x_2, x_3, x_4] < \epsilon
\end{equation*}
whenever $\epsilon > 0$ and $(x_1, x_2, x_3, x_4)$ is a four-tuple of distinct points in $X$.  Then $f$ is $\eta$-quasi-M\"obius with $\eta$ depending only on $\lambda, C_0$, and $\delta_0$.
\end{lemma}

To show that the function $\delta_0$ in the statement of \cite[Lemma 3.3]{BnK} exists, we argue by contradiction.  To use the $\wcaptwo$ condition, we must extract open sets from our data.  This is done by taking small open neighborhoods of the continua given by the following lemma.

\begin{lemma}[\cite{BnK}, Lemma 2.10]
Suppose $(Z,d)$ is $\lambda$-LLC.  Then there exist functions $\delta_1, \delta_2 \colon (0,\infty) \to (0,\infty)$ depending only on $\lambda$ with the following properties.  Suppose $\epsilon > 0$ and $(z_1, z_2, z_3, z_4)$ is a four-tuple of distinct points in $Z$. \\
\indent \text{\textnormal{(i)}} If $[z_1, z_2, z_3, z_4] < \delta_1(\epsilon)$, then there exist continua $E,F \subseteq Z$ with $z_1, z_3 \in E$, $z_2, z_4 \in F$, and $\Delta(E,F) \geq 1/\epsilon$. \\
\indent \text{\textnormal{(ii)}} If there exist continua $E,F \subseteq Z$ with $z_1, z_3 \in E$, $z_2, z_4 \in F$ and $\Delta(E,F) \geq 1 / \delta_2(\epsilon)$, then $[z_1, z_2, z_3, z_4] < \epsilon$.
\end{lemma}

One finds a contradiction by using the Loewner property in $\Stwo$ and the following lemma to appropriate sets.  Hence, our maps be uniformly quasi-M\"obius.

\begin{prop}[\cite{BnK}, Proposition 8.1]
Let $(Z, d, \mu)$, be a $Q$-regular metric measure space, $Q \geq 1$, and let $\mathscr{A}$ be a $K$-approximation of $Z$.  Then there exists a constant $C \geq 1$ depending only on $K$ and the data of $Z$ with the following property: \\
\indent If $E,F \subseteq Z$ are continua and if $\dist(V_E, V_F) \geq 4K$, then
\begin{equation*}
\qmod(E, F) \leq C \qmodg(V_E, V_F).
\end{equation*}
\end{prop}

Lastly, we may find a subsequence which converges to our desired map.

\begin{lemma}[\cite{BnK}, Lemma 3.1]
Suppose $(X, d_X)$ and $(Y, d_Y)$ are compact metric spaces, and $f_k \colon D_k \to Y$ for $k \in \N$ is an $\eta$-quasi-M\"obius map defined on a subset $D_k$ of $X$.  Suppose
\begin{equation*}
\lim_{k \to \infty} \sup_{x \in X} \dist(x, D_k) = 0
\end{equation*}
and that for $k \in \N$ there exist triples $(x_1^k, x_2^k, x_3^k)$ and $(y_1^k, y_2^k, y_3^k)$ of points in $D_k \subseteq X$ and $Y$, respectively, such that $f(x_i^k) = y_i^k$, $k \in \N$, $i \in \{1,2,3\}$,
\begin{equation*}
\inf \{ d_X (x_i^k, x_j^k) : k \in \N, i, j \in \{1,2,3\}, i \neq j\} > 0,
\end{equation*}
and 
\begin{equation*}
\inf \{ d_X (x_i^k, x_j^k) : k \in \N, i, j \in \{1,2,3\}, i \neq j\} > 0.
\end{equation*}
Then the sequence $(f_k)$ subconverges uniformly to an $\eta$-quasi-M\"obius map $f \colon X \to Y$, i.e. there exists an increasing sequence $(k_n)$ in $\N$ such that 
\begin{equation*}
\lim_{n \to \infty} \sup_{x \in D_{k_n}} d_Y(f(x), f_{k_n}(x)) = 0.
\end{equation*}
\end{lemma}

\subsection*{Proof of uniformization}

\begin{proof}[Proof of Theorem \ref{qs unif thm}]
Let $X$ be a hyperbolic filling of $Z$ with parameter $s$.  We apply \cite[Proposition 6.7]{BnK} to the vertices $V_X^n$ to form a graph $Y^n$, so $V_X^n \subseteq V_{Y^n}$ and $Y^n$ embeds as a triangulation of $Z$.  As in \cite{BnK}, this induces triangulations $T^n$ on $\Stwo$ with the same incidence pattern as $Y^n$ and for which the incidence pattern is realized as that of a normalized circle packing. This allows us to define functions $\phi_n \colon V_{T^n} \to Z$ where we map a vertex to its corresponding image point in the triangulation.  We note each $Y^n$ forms a $K$-approximation of $Z$ in the language of \cite{BnK} when associating the objects $p(y) = y$, $r(y) = s^{-n}$, and $U_y = B(p(y), Kr(y))$ by \cite[Corollary 6.8]{BnK} and $T^n$ forms a $K'$-approximation of $\Stwo$ with appropriately associated objects by \cite[Lemma 5.1]{BnK}.

We form an infinite graph $Y$ from the graphs $Y^n$ by connecting vertices $y \in Y^n$ and $y' \in Y^{n+1}$ if $y \in X^n$, $y' \in X^{n+1}$, and $y, y'$ are connected in $X$.  We imbue $Y$ with the graph metric.  

We claim $X$ and $Y$ are quasi-isometric.  Define $F \colon Y \to X$ on vertices by sending a given vertex $y \in Y^n$ to the closest vertex (in $V_{Y^n}$) to $y$ that belongs to $V_X^n$.  If $y,y'$ are neighboring vertices in $V_{Y^n}$ and $x, x' \in V_X^n$ are the closest vertices to $y,y'$, then by \cite[Proposition 6.7]{BnK} (iii) it follows that $d(x,x') < (2K+1)K2s^{-n}$.  Hence there is a constant $C(K)$ such that $(C(K)B_x) \cap B_{x'} \neq \emptyset$.  Thus, when $|y-y'|_Y = 1$ it follows that $|x-x'|_X$ is bounded by a constant depending only on $K$.  Now suppose the images $x, x' \in V_X^n$ of $y,y'$ satisfy $|x-x'|_X = 1$.  Then, $B_x \cap B_{x'} \neq \emptyset$, so $d(y,y') < 6s^{-n}$.  As noted in \cite{BnK}, from (iii) we conclude $|y-y'|_Y \leq C(6,K,C_0,\lambda)$.  Thus, there is a constant $C > 0$ such that for all $y, y' \in Y^n$, we have 
\begin{equation*}
\frac{1}{C}|y-y'|_Y-C \leq |x-x'|_X \leq C|y-y'|_Y + C.
\end{equation*}
This holds for all $y \in Y$ as the only connections between levels are those already there from $X$ and $F(x) = x$ whenever $x \in V_X$.

For the coarse inverse $G$, we set $G(x) = x \in Y$.  It follows from \cite[Proposition 6.7]{BnK} (iii) that $G \circ F$ and $F \circ G$ are within bounded distance to the identity and that every point in $Y$ is within $K$ of a point in $V_X$.

We claim the family $\{\phi_n\}$ restricted to suitable subsets of $V_{Y^n}$ is uniformly quasi-M\"obius.  For this we apply \cite[Lemma 4.7]{BnK} to $Y^n$ with $L = K$ to generate $W = W^n$, a set of vertices such that $p(W)$ is weakly uniformly perfect.  We restrict $\phi_n$ to these subsets, which we denote as $\phi_n^W$.  Suppose that $\phi_n^W$ is not uniformly quasi-M\"obius.  Then, the condition given in \cite[Lemma 3.3]{BnK} must fail.  Hence, there exists $\epsilon > 0$ such that for each $\delta_k = 1/k$ there exists $\phi_{n_k}$ and four points $x_{1,k}, x_{2,k}, x_{3,k}, x_{4,k} \in W^{n_k}$ such that
\begin{equation*}
[x_{1,k}', x_{2,k}', x_{3,k}', x_{4,k}'] < \delta_k
\end{equation*}
but
\begin{equation*}
[x_{1,k}, x_{2,k}, x_{3,k}, x_{4,k}] \geq \epsilon
\end{equation*}
where $x'$ denotes the image of $x$ under $\phi_{n_k}^W$.  Note we may assume $n_k \to \infty$.
Applying \cite[Lemma 2.10]{BnK} (i) we see that if $\delta_k$ is small enough, then there are continua $E_k'$ and $F_k'$ in $Z$ such that $x_{1,k}', x_{3,k}' \in E_k'$, $x_{2,k}', x_{4,k}' \in F_k'$, and $\Delta(E_k',F_k')$ is large quantitatively.  We consider open neighborhoods $N_\rho(E_N')$ and $N_\rho(F_N')$ where $N$ is large enough and $\rho$ is small enough such that $\Delta(N_\rho(E_N'), N_\rho(F_N'))$ is large and hence  $\varphi_2(\Delta(N_\rho(E_N'), N_\rho(F_N')))$ is small.  Fix $N$ and $\rho$.

For ease of notation, we write $x_i = x_{i,N}$ and $x_i' = x_{i,N}'$.  Consider $V_{E_j'} \subseteq V_{Y^j}$ defined by $V_{E_j'} = V_{E_j'}(N) = \{y \in Y^j : U_y \cap E_N' \neq \emptyset\}$ and likewise for $V_{F_j'}$.  From the corresponding vertices in $T^j$ we define $V_{E_j}$ and $V_{F_j}$.  From the incidence pattern in $V_{E_j}$ and $V_{F_j}$ and the fact that $Y^j$ embeds in $Z$ we extract continua $E_j$ and $F_j$ such that $x_1, x_3 \in E_j$ and $x_2, x_4 \in F_j$.  From this it follows that $\Delta(E_j, F_j) \leq \langle x_1, x_2, x_3, x_4 \rangle^{-1}$.  As $\Stwo$ is a Loewner space, it follows that there is a lower bound $0 < a \leq \modtwo(E_j, F_j)$.  For large $j$, the combinatorial distance between $E_j$ and $F_j$ is large in $T^j$ as these sets were constructed from $Y^j$ and the sets there have large combinatorial distance.  Thus, we may apply \cite[Proposition 8.1]{BnK} to conclude there is a constant $c > 0$ depending only on $[x_1, x_2, x_3, x_4]$ such that for large $j$ we have $c < \modtwo(V_{E_j}, V_{F_j}, T^j)$.  As the graphs are the same, this means for large $j$ we have $c < \modtwo(V_{E_j'}, V_{F_j'}, Y^j)$.

Let $\tau \colon E_X \to [0,1]$ be admissible for $\wcaptwo(N_\rho(E_N'), N_\rho(F_N'))$ and satisfy $\norm{\tau}_{2,\infty} <  \infty$.  From Lemma \ref{Last Path Lemma}, for large enough levels $j$ from any vertex in $V_{E_j'}$ we can guarantee the existence of a path connecting down to $N_\rho(E_N')$ with $\tau$-length bounded above by $1/4$ (just avoid the required large values $\beta$ for enough levels; these occur on finitely many levels).  Then, $f_j \colon V_{Y^j} \to [0,\infty)$ defined by $f_j(v) = \sum_{e \sim v} 2\tau|_j(e)$ must be admissible for $\modtwo(N_\rho(E_N'),N_\rho(F_N'),Y^j)$. The norms $\norm{f_j}_2$ and $\norm{2\tau|_j}_2$ are comparable, so we conclude that there is a constant $c' > 0$ depending only on $c$ such that $\norm{\tau|_j}_2^2 \geq c'$ for these levels.  

From \cite[Proof of Theorem 1.4]{BnS}, for any $N'>1$ there is an $n \in [N',2N']$ such that 
\begin{equation*}
\norm{\tau|_n}_2^2 = \sum_{e \in E_{Y^n}} \tau(e)^2 \lesssim \norm{\tau}_{2,\infty}^2.
\end{equation*}
but the left hand side of this is bounded below by $c' > 0$ for large $n$ and the right hand side tends to $0$ with different choices of $N$ and $\rho$.  Thus, our maps $\phi_n^W$ are uniformly quasi-M\"obius.

Lastly, we apply \cite[Lemma 3.1]{BnK} to the sequence $\phi_n^W$.  Because we normalized our circle packings, we need only know that the subsets $W^n$ have the property that 
\begin{equation*}
\lim_{n \to \infty} \sup_{x \in X} \dist(x,W^n) = 0.
\end{equation*}
The vertices in $W^n$ are a maximal combinatorially $K$-separated subset of $Y^n$.  Thus, by \cite[Proposition 6.7]{BnK} every vertex in $y \in V_{Y^n}$ satisfies $d(y, W^n) < K^2 s^{-n}$.  In particular, every $x \in V_X^n$ has this property and every point $z \in Z$ is $s^{-n}$ close to some point $x \in V_X^n$, so every $z \in Z$ is within $(K^2 + 1) s^{-n}$ of some point in $W^n$.  

As quasi-M\"obius maps on bounded spaces are quasisymmetric, the result follows.
\end{proof}

We now prove Corollary \ref{local QS}.

\begin{proof}[Proof of Corollary \ref{local QS}]
Let $X$ be a hyperbolic filling of $Z$.  By Theorem \ref{qs unif thm} it suffices to show there is a constant $C > 0$ and a value $t_0 > 0$ such that if $\Delta(A,B) = t_0$, then $\wcaptwo(A,B) < C$.  

\sloppypar
For each $z \in Z$, there is an $r_z > 0$ and an increasing function $\eta_z \colon (0,\infty) \to (0,\infty)$ with $\eta_z(t) \to 0$ as $t \to 0$ such that there is an $\eta_z$-quasisymmetry $f_z \colon \overline{B(z,r_z)} \to V_z \subseteq \Stwo$, where $\overline{B(z,r_z)}$ is the closure of $B(z,r_z)$ in $Z$.  We may choose $r_z$ such that $\diam(V_z) < \diam(\Stwo)/10$.  Set $U_z = B(z, r_z/2)$.  The collection $\{U_z\}_{z \in Z}$ covers $Z$, so there is a finite subcover $U_{z_1}, \dots, U_{z_N}$ of $Z$.  For ease of notation, we will replace $z_i$ by $i$ in indexing subscripts.  By setting $\eta = \max_{1 \leq i \leq N} \{\eta_i\}$, we see each $\overline{B(z_i,r_i)}$ is $\eta$-quasisymmetrically homeomorphic to $V_i$.  Let $\delta > 0$ be such that for every $z \in Z$, there is an $i$ such that $B(z, \delta) \subseteq U_i$ (this is given by the Lebesgue number; see \cite[p179]{Mu}).  Suppose $\Delta(A,B) = t_1$.  We may assume $\diam(A) \leq \diam(B)$.  Then, $\dist(A,B) = t_1 \diam(A)$, so $\diam(A) \leq 1 / t_1$.  Choose $t_1$ large enough that $\diam(A) < \delta$, so we may assume $A \subseteq U_i$ for some $i$.  Let $A' = f_i(A)$  Let ${B_0 = (\{z : d(z,z_i) > 3r_i / 4\} \cup B) \cap B(z_i, r_i)}$ and let $B' = f_i(B_0) \cup (\Stwo \setminus V_i)$.  Then, $A', B' \subseteq \Stwo$ are open.  As $\diam(V_i) < \diam(\Stwo)/10$, we see ${\Delta(A', B') = \dist(A', B') / \diam(A')}$.  As $\{z : d(z,z_i) > 3r_i / 4\} \subseteq f_i^{-1}(B')$, it follows that $\dist(A',B') = \dist(A', B' \cap V_i)$.  As we have a uniform $\eta$, by Lemma \ref{qs rel dist lemma} applied to the spaces $\overline{B(z_i,r_i)}$ and $V_z$ it follows that, given $t_2 > 0$, if $t_1$ is large enough then $\Delta(A', B') \geq t_2$.

Let $Y$ be a hyperbolic filling for $\Stwo$.  The proof of (ii) $\implies$ (iii) in Theorem \ref{varphi equivalence} shows there is a $t_3 > 0$ and an $R(t_3) > 0$ such that if $\Delta(A', B') \geq t_3$, then there is an admissible vertex function $\tau' \colon V_Y \to [0,1]$ for $\wcaptwo(A', B')$ with $\norm{\tau'}_{2,\infty}^2 \leq 1$ that is supported on $H(B(a',R))$ (the hull of $B(a',R)$, see Definition \ref{hull def}) for some $a' \in A'$.  From Remark \ref{T support}, $R = c_3^N \diam(A')$.  Let $D > 0$.  As ${f(\{z : d(z,z_i) > 3r_i / 4\}) \subseteq B'}$, if $t_3$ is large enough, then ${\rho_Y(H(B(a,R)),Y \setminus H(V_z)) \geq D}$.  

Let $W_X = \{v \in V_X : B_v \subseteq B(z_i, r_i)\}$.  Recall we define the map $f_i|_{W_X}$ as a vertex map from $W_X$ into $V_Y$ by finding the highest level $w$ in $Y$ such that $f_i(B_v) \subseteq B_w$ and setting $f_i|_{W_X}(v) = w$.  It follows that $f_i|_{W_X}$ is a quasi-isometry with its image with constants depending only on $\eta$.  In particular, there is a $D > 0$ such that if $v, v' \in W_X$ satisfy $|v-v'| = 1$, then $|f(v) - f(v')| \leq D$ in $V_Y$.  Define $\tau \colon V_X \to [0,\infty]$ by setting, for $v \in W_X$,
\begin{equation*}
\tau(v) = \sum_{|f_i|_{W_X}(v) - y| \leq D} \tau'(y)
\end{equation*}
and setting $\tau(v) = 0$ otherwise.

We claim $\tau$ is admissible for $\wcaptwo(A,B)$.  Let $\gamma$ be a vertex path connecting $A$ and $B$.  Let $\gamma_W$ be the subpath formed from vertices in $\gamma$ that on one side converge to $A$ and on the other side continue until we reach the first vertex outside of $W_X$.  Denote this $\{v_n\}$.  We create image path $\gamma'_W$ in $V_Y$ by connecting $f_i|_{W_n}(v_n)$ and $f_i|_{W_n}(v_{n+1})$ by a path $\sigma_n$ in $V_Y$ of length at most $D$ for each $n$ where this makes sense.  Then, 
\begin{equation*}
\sum_\gamma \tau(v) \geq \sum_{\gamma_W} \tau(v) = \sum_n \tau(v_n) = \sum_n \sum_{|f_i|_{W_X}(v_n) - y| \leq D} \tau'(y) \geq \sum_n \sum_{y \in \sigma_n} \tau'(y).
\end{equation*}
Now, if $\gamma$ leaves $W_X$, then the paths $\sigma_n$ leave $H(B(a',R))$.  In this case, by the definitions of $W_X$ and $B(a', R)$ it follows that we may continue the paths to connect to $B'$ outside of $H(B(a',R))$, meaning by the admissibility of $\tau'$ that we must have $\sum_{v \in \gamma} \tau(v) \geq 1$.  Otherwise, $\gamma$ connects $A$ and $B$ in $W_X$ and so the path $\gamma'_W$ connects $A', B'$ in $V_Y$, so again by admissibility of $\tau'$ we have $\sum_{v \in \gamma} \tau(v) \geq 1$.  Either way, $\tau$ is admissible for $\wcaptwo(A,B)$.

The norm comparison between $\tau$ and $\tau'$ follows as in \cite[Theorem 1.4]{L1}, the important thing to note is the quasi-isometry data depends only on $\eta$.
\end{proof}

\fussy
\section{Other exponent results}\label{Section other thms}


\subsection*{Group action equivalence}

We now show that the presence of a cocompact group action is strong enough to conclude $Q_w = Q_w'$.  Recall our setting: $(Z, d, \mu)$ is a compact, connected, Ahlfors regular metric space with hyperbolic filling $X$ and there exists a metric space $Y$ quasi-isometric to $X$ and a subgroup $G \subseteq \Iso(Y)$ acting cocompactly on $Y$.  We first prove some preparatory lemmas.

\begin{lemma}\label{path dist lemma}
Let $\gamma \colon \N_0 \to V_X$ be a path converging to a point $z \in Z$ such that there is a function $f$ such that for all $k$ we have $\ell(\gamma(k)) \geq f(k)$.  Then, for all $k$ we have $d(\gamma(k),z) \leq 4s \sum_{j=k}^\infty s^{-f(j)}$.
\end{lemma}

\begin{proof}
As $\gamma$ is a path, we have $d(\gamma(j), \gamma(j+1)) \leq 4s^{-f(j) + 1}$.  Thus, 
\begin{equation*}
d(\gamma(k),z) \leq \sum_{j=k}^\infty 4s^{-f(j) + 1} = 4s \sum_{j=k}^\infty s^{-f(j)}
\end{equation*}
\end{proof}

\sloppypar
\begin{lemma}\label{QI dist lemma}
Suppose that there are constants $\alpha, \beta, D > 0$ such that for every $v \in X$, there is an $(\alpha, \beta)$ quasi-isometry $f_v \colon V_X \to V_X$ with $|f_v(v) - O| \leq D$.
Then, for each $\epsilon > 0$ there exists $\delta_0 > 0$ with the following property.  Let
$A \subseteq Z$ be open, $\diam(A)  = \delta < \delta_0$.  Then, there are $t_0(\delta_0), t_1(\delta)$ such that if $t \in [t_0, t_1]$, $a \in A$, and $B = {Z \setminus \overline{B(a, t \delta)}}$, then $B \neq \emptyset$ and there exists $v \in V_X$ such that if $F$ is the quasisymmetry on $Z$ induced by $f_v$ and $A' = F(A), B'=F(B)$, then $\diam(A'), \diam(B') < \epsilon$.  Moreover, there is a constant $c(\alpha, \beta, X)$ such that $\dist(A', B') \geq c$.
\end{lemma}

\fussy
\begin{proof}
We assume $A, B, \delta$ are defined as in the statement of the theorem.  We define $t_1(\delta) = s^{-10} / \delta$, from which is follows that $B \neq \emptyset$ as $\diam(Z) = 1$.  Thus, it suffices to establish $\delta_0$ and $t_0$.  We note this choice of $t_1$ also guarantees $O \notin H(B(a, t_1 \delta))$.

We define $M = M(\delta, t)$ as follows.  Consider $H(A)$ and $V_X \setminus H(B)$.  Fix $a \in A$.  We can find a geodesic ray $\gamma_A$ in $X$ beginning at $O$ and with limit $a$, so $\gamma_A(k) \in V_X$ for all $k \in \N$.  Let $\ell_B = \sup \{\ell(\gamma_A(k)) : \gamma_A(k) \notin H(B(a, t\delta))\}$ and $\ell_A = \inf \{\ell(\gamma_A(k)) : \gamma_A(k) \in H(A)\}$.  Set $M = \ell_A - \ell_B$.  It follows that $M(t,\delta) \simeq \log_s(2t)$.  Set $t_0(\delta_0) = t_1(\delta_0) / 2 = s^{-10} / (2 \delta_0)$.  Then, if $\delta_0$ is small, $M(t, \delta) \gtrsim \log_s(s^{-10} / \delta_0)$ is large.  Hence, to prove the claim it suffices to show the desired properties hold when $M$ is large enough.

\sloppypar
We can find a vertex $v_m \in \gamma_a$, with $\ell(v_m) = m$, that has distance at least $M/3$ from both $H(A)$ and $V_X \setminus H(B(a, t\delta)))$.  We assume $M$ is large enough that $A \subseteq B_{v_m}$.  Let $f = f_{v_m}$ be the quasi-isometry guaranteed by the hypothesis.  We wish to estimate $\diam(A'), \diam(B'), \dist(A', B')$.  We recall the distance on $Z$ is visual, so $d(z, z') \simeq s^{-(z|z')_O}$ whenever $z,z' \in Z$.  As $|f(v_m) - O| \leq D$, we may bound  $(\cdot|\cdot)_{f(v_m)}$ instead of $(\cdot|\cdot)_O$ for our estimates.   To do this, it suffices to bound $(\cdot | \cdot)_{v_m}$ by \cite[Proposition 5.15]{GH}.  We will use \cite[Proposition 2.17]{GH}, which says if $v, w \in V_X$ and $[v,w]$ is a geodesic connecting $v$ to $w$, then $(v|w)_{v_m} = \rho(v_m, [v,w])$ up to an additive factor.  Recall we use $F$ to denote the quasisymmetry induced by $f_{v_m}$.

\fussy

First, let $F(a), F(a') \in A'$.  Then, there are geodesics $\gamma_a, \gamma_{a'}$ with limits $a$ and $a'$ in $X$.  By construction, the ball $B_{\gamma_A(\ell_A - 3)}$ contains $A$, so we may assume both $\gamma_a$ and $\gamma_{a'}$ contain $\gamma_A(\ell_A - 3)$.  As $(a|a')_{v_m} = \lim_{k \to \infty} (\gamma_a(k)|\gamma_{a'}(k))_{v_m}$ up to an additive factor, we estimate the quantities $ (\gamma_a(k)|\gamma_{a'}(k))_{v_m}$.  We have, for $k \geq \ell_A$,
\begin{equation*}
\begin{split}
2(\gamma_a(k)|\gamma_{a'}(k))_{v_m} &= |v_m - \gamma_a(k)| + |v_m - \gamma_{a'}(k)| - |\gamma_a(k) - \gamma_{a'}(k)| \\
&\geq 2(k-m) - 2(k - (\ell_A - 3)) \\
&= 2(\ell_A - 3 - m).
\end{split}
\end{equation*}
This last quantity is at least $M/4$ for large enough $M$.  It follows that $\diam(A') \to 0$ as $M \to \infty$.

Now, let $F(b), F(b') \in B'$.  We bound $-(b|O)_{v_m}$ and use the $\delta_G$-inequality $-(b|b')_{v_m} \leq -(b|O)_{v_m} \vee (-b'| O)_{v_m} + \delta_G$, where $\delta_G$ is the Gromov hyperbolicity constant for $X$.  Let $\gamma_b$ be a geodesic ray in $X$ that begins at $O$ and has limit $b$.  We may assume $b \in B_{\gamma_b(k)}$ for all $k$.  Note that the level of $\gamma_b(k)$ is $k$.

To bound $-(b|O)_{v_m}$, we bound $(\gamma_b(k)|O)_{v_m}$ for large $k$.
Let $\sigma \colon [0,L] \to X$ be a minimal geodesic from $v_m$ to $\gamma_b$.  By \cite[Lemma 2.17]{GH} it suffices to bound $L$ below.  As $b \in B_{\sigma(L)}$, we have 
\begin{equation*}
d(v_m, b) \leq \sum_{j=0}^{L} \diam(B_{\sigma(j)}) \leq \sum_{j=0}^{L} 4 s^{-m + j} \leq 4(L+1) s^{-m + L}.
\end{equation*}
As $A \subseteq B_{v_m}$, it follows that
\begin{equation*}
(t-1)\diam(A) = (t-1)\delta \leq d(A, b) \leq 4(L+1) s^{-m + L}.
\end{equation*}
Since $v_m$ has distance at least $M/3$ from $V_X \setminus H(B(a, t\delta)))$, we see $s^{-m}s^{M/3} \leq \diam(B(a, t\delta))) \leq 2t\delta$.  Combining these estimates, we have, for $t \geq 2$, 
\begin{equation*}
s^{M/3} \leq 16(L+1)s^L.
\end{equation*}
It follows that $\diam(B') \to 0$ as $M \to \infty$.

For $\dist(A', B')$, also compute from $v_m$.  Let $a \in A$ and let $b \in Z$ be such that $d(a,b) \geq 1/3$.  The point $b$ exists as $\diam(Z) = 1$.  It follows that $b \in B$.  Consider a geodesic $\sigma$ from $b$ to $a$ in $X$.  By \cite[Lemma 2.17]{GH}, there is a constant $C$ depending only on $X$ such that $\rho(O, \sigma) \leq C$. Consider a vertex $w \in \sigma$ of minimal level, so $\ell(w) \leq C$.  Consider the subpath $\gamma$ which is the portion of $\sigma$ from $\gamma(0) = w$ to $a$.  We see $|\gamma(i) - O| \geq |\gamma(i) - \gamma(0)| - |\gamma(0) - O|$ from which it follows that $\ell(\gamma(i)) \geq i - C$.  Thus, by Lemma \ref{path dist lemma}, $d(\gamma(k), a) \leq 4s \sum_{j=k}^\infty s^{-(j-C)}$.  By construction, $\ell(\gamma(i_0)) = m$ for some $i_0$.  We see $m - C \leq i_0 \leq m + C$.  Then, by Lemma \ref{path dist lemma}, 
\begin{equation*}
d(\gamma(i_0),a) \leq 4s \sum_{j=i_0}^\infty s^{-(j-C)} \leq 4s^{C+1} \sum_{j=i_0}^\infty s^{-j} \leq 4s^{C+1 - i_0} \frac{s}{s-1}.
\end{equation*}
This is bounded above by $C_1 s^{-m}$ where $C_1$ is a constant that only depends on $C$ and $s$.  Let $j_0$ be such that $s^{j_0} > C_1 + 2$.  By the construction of $X$, there is a vertex $v$ of level $m - j_0$ with $d(v, \gamma(i_0)) \leq s^{-m+j_0}$.  We see $\gamma(i_0) \in B_v$.  Now,
\begin{equation*}
\begin{split}
d(v_m, v) &\leq d(v, \gamma(i_0)) + d(\gamma(i_0), a) + d(a , v_m) \\
&\leq s^{-m+j_0} + C_1 s^{-m} + 2s^{-m} \\
&< 2s^{-m + j_0}.
\end{split}
\end{equation*}
Thus, $v_m \in B_v$, so there is a path from $v$ to $v_m$ of length $j_0$ in $X$.  Hence, ${|v_m - \gamma(i_0)| \leq 2j_0}$.  From this we see $(a|b)_m \leq 2j_0$ up to an additive constant.  For this to work, we needed $m-j_0 \geq 0$, which is possible for large $m$ as $j_0$ only depends on $s$ and $X$.  We conclude $\dist(A',B') \gtrsim s^{-2j_0}$ with implicit constants arising from the quasi-isometry constants of $f_{v_m}$ and the visual metric parameters.
\end{proof}

\sloppypar
\begin{proof}[Proof of Theorem \ref{group action thm}]
Consider the space $Z \times Z$ equipped with the product metric
\begin{equation*}
d_2((z,w),(z',w')) = (d(z,w)^2 + d(z', w')^2)^{1/2}.
\end{equation*}
Let $\epsilon > 0$ and let $U_\epsilon \subseteq Z \times Z$ be the $\epsilon$-neighborhood of the diagonal ${\{(z,z) : z \in Z\} \subseteq Z \times Z}$.  Then, $Z \times Z \setminus U_\epsilon$ is compact.  Let $a,b \in Z$ be such that $(a,b) \notin U_\epsilon$.  Then, there are positively separated open sets $A_a, B_b$ with $a \in A_a$ and $b \in B_b$.  The collection $\{A_a \times B_b\}$ forms an open cover of $Z \times Z \setminus U_\epsilon$, so we may find a finite subcover $\{A_i \times B_i\}_{i=1}^I$.  It follows that there is a $\delta > 0$ such that whenever $U \subseteq Z \times Z$ satisfies $\diam(U) < \delta$, then there is some $i$ with $U \subseteq A_i \times B_i$ (\cite[p179]{Mu}).  From this, for each $c > 0$ there is an $\epsilon' > 0$ such that if $\diam(A'), \diam(B') < \epsilon'$ and $\dist(A',B') \geq c$, then there is some $i$ such that $A' \times B' \subseteq A_i \times B_i$.

\fussy
We will use Lemma \ref{QI dist lemma}.  Let $\phi \colon X \to Y$ be a quasi-isometry with coarse inverse $\phi^{-1}$.  We assume $\phi^{-1}(Y) \subseteq V_X$.  Let $K \subseteq Y$ be a compact subset such that for every $y \in Y$, there exists $g_y \in G$ with $g_yy \in K$.  We assume $\phi(O) \in K$.  It follows that there is a constant $C_1 > 0$ such that every $y \in K$ has distance at most $C_1$ to $\phi(O)$.  For $v \in V_X$, define $\phi_v \colon V_X \to V_X$ by $\phi_v(x) = \phi^{-1}(g_{\phi(v)} \phi(x))$.  As $\phi$ is fixed and $g_y$ is an isometry for each $y$, the quasi-isometries $\phi_v$ satisfy the hypothesis of Lemma \ref{QI dist lemma} with $D = C_1$.  

From the quasi-isometry constants of $\{\phi_v\}$, we obtain $c > 0$ as in Lemma \ref{QI dist lemma}.   As above, we can then find $\epsilon' > 0$ such that if $\diam(A'), \diam(B') < \epsilon'$ and $\dist(A',B') > c$, then there is some $i$ such that $A' \times B' \subseteq A_i \times B_i$.  Applying Lemma \ref{QI dist lemma} with this $\epsilon'$, we obtain $\delta_0$ with the properties as stated in the lemma. 

Let $p > Q_w$.   Let $C = \max_i \pcap(A_i, B_i)$.  Let $t' = 2 / \delta_0$.  Suppose $A,B \subseteq Z$ are open and $\Delta(A, B) = t'$.  We may assume $\diam(A) \leq \diam(B)$.  Let $\delta = \diam(A)$.  Then, as $\diam(Z) = 1$, we have $\delta < \delta_0$.  

Suppose first that $\delta \leq s^{-10} \delta_0 / 2$.  Then, $t' \leq s^{-10}/\delta = t_1(\delta)$ and $s^{-10} / (2\delta_0) = t_0 \leq t'$.  Let $t'' = (t_0 + t')/2$.  For $a \in A, b \in B$ we have $d(a,b) \geq t' \delta$.  It follows that for fixed $a \in A$, we have $B \subseteq Z \setminus \overline{B(a, t''\delta)}$.  Thus, by Lemma \ref{QI dist lemma}, there is a quasisymmetry $F$ from $Z$ to $Z$ with $A' = F(A), B' = F(B)$ satisfying $\diam(A'), \diam(B') < \epsilon'$ and $\dist(A', B') \geq c$.  We note the control function $\eta$ for $F$ depends only on the quasi-isometry constants of $\{\phi_v\}$.  Thus, there exists some $i$ such that $A' \times B' \subseteq A_i \times B_i$.  That is, $A' \subseteq A_i$ and $B' \subseteq B_i$ so, by Theorem \ref{QS inv qcap}, $\pcap(A,B) \lesssim C$ with implicit constant not depending on $A, B$.

Now, suppose $\delta > s^{-10} \delta_0 / 2$.  Find $k$ such that $2^{-k} \delta \leq  s^{-10} \delta_0 / 2$.  As $Z$ is doubling, there exists $N(k)$ sets $A_j^k$ of diameter at most $2^{-k} \delta$ with $A \subseteq \cup_j A_j^k$.  As $Z$ is connected, we may assume $\diam(A_j^k) \geq 2^{-(k+1)} \delta$.  For each $A_j^k$, we have $\dist(A_j^k, B) \geq t'\delta - 2^{-k}\delta$.  As $k \geq 1$ and $\delta_0 \leq 1$, we see $2^{k+1} - 2 > \delta_0$ from which it follows that $(t'-2^{-k})\delta > t'2^{-k}\delta$.  Hence,
\begin{equation*}
\dist(B, A_j^k) \geq (t'-2^{-k})\delta > t'2^{-k}\delta \geq t' \diam(A_j^k)
\end{equation*}

Thus, if $a_j^k \in A_j^k$, we see 
\begin{equation*}
B \subseteq Z \setminus \overline{B(a_j^k, t' \diam(A_j^k))}.
\end{equation*}
As $t_0 \leq t'$ and 
\begin{equation*}
t' = 2/\delta_0 \leq 2^k s^{-10} / \delta \leq s^{-10} / \diam(A_j^k) = t_1(\diam(A_j^k)),
\end{equation*}
we may apply Lemma \ref{QI dist lemma} to conclude that there is a quasisymmetry $F$ from $Z$ to $Z$ with $A' = F(A_j^k), B' = F(B)$ satisfying $\diam(A'), \diam(B') < \epsilon'$ and $\dist(A', B') \geq c$.  It follows as before that $\pcap(A_j^k,B) \lesssim C$, so by subadditivity we have $\pcap(A,B) \lesssim N(k)C$.

As $k$ only depended on $\delta_0$, it follows that $\varphi_p(t') < \infty$.  Hence, $Q_w' \leq p$ and, as $p > Q_w$ was arbitrary, we see $Q_w' \leq Q_w$.  In the above, all we needed was the finiteness of $\pcap(A_i, B_i)$ for the finite number of pairs of sets $(A_i, B_i)$, so it follows that if the infimum in the definition of $Q_w$ is attained, then it is for $Q_w'$ as well.
\end{proof} 

\begin{remark}
Corollary \ref{QS unif cor} is immediate from the observations at the end of the above proof.
\end{remark}

\subsection*{CLP exponent equivalence}

In \cite{BdK} the Combinatorial Loewner Property (CLP) is investigated.  This property is a discretized version of the Loewner property discussed in Theorem \ref{qcap < qmod}; it requires that appropriately defined discrete modulus of connecting path families is controlled both below and above by relative distance.  We show that if an Ahlfors $Q$-regular metric space has the CLP with exponent $Q$, then $Q_w = Q_w' = \ARCdim = Q$.  First we define some terms from \cite{BdK} to make this precise.

We work in $(Z,d,\mu)$ a compact, connected Ahlfors $Q$-regular metric measure space.  For $k \in \N$ and $\kappa \geq 1$, a finite graph $G_k = (V_k, E_k)$ is a {\em $\kappa$-approximation on scale $k$} if it has the following properties: each $v \in V_k$ corresponds to an open set $U_v \subseteq Z$, we have $Z = \cup_{v \in V} U_v$, vertices $v$ and $w$ are connected by an edge if and only if $U_v \cap U_w \neq \emptyset$, there is an $s > 1$ such that for every $v \in V_k$ there is a $z_v \in Z$ with 
\begin{equation*}
B(z_v, \kappa^{-1} s^{-k}) \subseteq U_v \subseteq B(z_v, \kappa s^{-k})
\end{equation*}
and, if $v,w \in V_k$ are distinct, then 
\begin{equation*}
B(z_v, \kappa^{-1}s^{-k}) \cap B(z_w, \kappa^{-1}s^{-k}) = \emptyset.
\end{equation*}

In \cite{BdK} the value $s = 2$ is used.  Let $\rho \colon V_k \to [0,\infty)$.  Given a path $\gamma \subseteq Z$, the function $\rho$ is admissible for $\gamma$ if 
\begin{equation*}
\sum_{U_v \cap \gamma \neq \emptyset} \rho(v) \geq 1.
\end{equation*}
If $\Gamma$ is a path family in $Z$, we say $\rho$ is admissible for $\Gamma$ if it is admissible for all $\gamma \in \Gamma$.  The {\em $G_k$-combinatorial p-modulus} of the path family $\Gamma$ is defined by
\begin{equation*}
\Modp(\Gamma, G_k) = \inf \{\sum_{v \in V_k} \rho(v)^p : \rho \text{ admissible for } \Gamma \}.
\end{equation*}
Given two sets $A,B \subseteq Z$, we write $\Modp(A,B,G_k)$ for the $G_k$-combinatorial $p$-modulus of the path family connecting $A$ and $B$.
\begin{defn}
\label{CLP def}
The metric space $Z$ satisfies the {\em Combinatorial Loewner Property with exponent $p>1$} if there exist decreasing functions $\phi, \psi \colon (0,\infty) \to (0,\infty)$ with $\lim_{t \to \infty} \phi(t) = \lim_{t \to \infty} \psi(t) = 0$ such that for all disjoint continua $A,B \subseteq Z$ and all $k$ with $s^{-k} \leq \min\{\diam(A),  \diam(B)\}$ we have
\begin{equation*}
\phi(\Delta(A,B)) \leq \Modp(A,B,G_k) \leq \psi(\Delta(A,B)).
\end{equation*}
\end{defn}

\begin{proof}[Proof of Theorem \ref{CLP theorem}]
Let $X$ be a hyperbolic filling for $Z$ with parameter $s$.  Fix $1 < p < Q$.  From the monotonicity of $\pcap$ we may assume $A = B(a,r_a)$ and $B = B(b, r_b)$ are balls.  Let $\phi$ be as in Definition \ref{CLP def}.  We note that $Z$ is linearly connected by \cite[Proposition 2.5]{BdK}.  Hence, there is a path in $Z$ connecting $a$ and $b$ (with controlled diameter, but we just need the path).  From this path, we find continua $E \subseteq \frac{1}{2} A$ and $F \subseteq \frac{1}{2} B$.

Suppose there exists $\tau \colon V_X \to [0,1]$ which is admissible for $\pcap(A,B)$ with $\norm{\tau}_{p,\infty} < \infty$.  It follows that $\tau \in \ell^Q(V_X)$ as $p < Q$.  Thus, there is a level $N_1$ such that for all $n \geq N_1$ we have $\norm{2\tau|_n}_Q^Q < \phi(\Delta(E,F))$, where $2\tau|_n$ denotes the restriction of $\tau$ to vertices on level $n$.  For large such levels, it follows that $2\tau|_n$ is not admissible for $\Modp(E,F,X^n)$, where $X^n$ is the subgraph of $X$ consisting of vertices on level $n$ and the edges in $E_X$ which connect such vertices.  We see $X^n$ is a $\kappa$-approximation on scale $n$ with $\kappa = 2$.  Hence, there is an $N_2 \geq N_1$ such that if $n \geq N_2$, then there exists a path $\gamma_n$ in $Z$ connecting $E$ and $F$ with
\begin{equation*}
\sum_{v \in V_{\gamma_n}} \tau(v) < 1/2
\end{equation*}
where $V_{\gamma_n} = \{v \in X^n : B_v \cap \gamma_n \neq \emptyset\}$.

Choose $\beta$ such that $S(\norm{\tau}_{p,\infty}, \beta) < 1/5$ from Lemma \ref{First Path Lemma} (where we use $S$ as the corresponding bound for vertex ascending paths; see Remark \ref{vertex edge comp}).  As $\tau \in \ell^{p,\infty}(V_X)$, it follows that there is a level $N_3(\beta) \geq N_2$ such that for all vertices on levels $\geq N_3$ we have $\tau(v) < \beta$.    As $E \subseteq \frac{1}{2} A$ and $F \subseteq \frac{1}{2} B$ there is an $N_4 \geq N_3$ such that if $n \geq N_4$ and $v \in V_{X^n}$ satisfies $B_v \cap E \neq \emptyset$ or $B_v \cap F \neq \emptyset$ then $B_v \subseteq A$ or $B_v \subseteq B$.  Thus, concatenating the edge path $\gamma_n'$ provided $n \geq N_4$ with those guaranteed from Lemma \ref{First Path Lemma} produces a path $\gamma$ connecting $A$ and $B$ with 
\begin{equation*}
\sum_{v \in \gamma} \tau(v) < 1/2 + 1/5 + 1/5 < 1,
\end{equation*}
contradicting the admissibility of $\tau$.
\end{proof}

\begin{remark}
In the above setting, $Q_w = Q_w' = \ARCdim = Q$.  Indeed, Theorem \ref{CLP theorem} shows that if $p < Q$, then $p \leq Q_w$.  Thus, $Q \leq Q_w$.  As $Z$ is $Q$-regular, it follows that $\ARCdim \leq Q$.  Hence, $Q \leq Q_w \leq Q_w' \leq \ARCdim \leq Q$.
\end{remark}


\subsection*{When the Ahlfors regular conformal dimension is attained}

To prove Theorem \ref{Wk Tan Thm} we assume the metric space $(Z,d)$ attains its Ahlfors regular conformal dimension.  In this case we will show $Q_w' = \ARCdim$.  To do so, we will work with a curve family of positive modulus in a weak tangent of $Z$.  The construction of what we will use appears in \cite{CP}, while the original positive modulus result appears in \cite{KL}.  

We first define weak tangents (cf. \cite[Definition 3.4]{CP}).

\begin{defn}
Compact pointed metric spaces $(Z_k, d_k, \mu_k, z_k)$ converge to the pointed metric space $(Z_\infty, d_\infty, \mu_\infty, z_\infty)$ if there is a pointed metric space $(W, d_W, q)$ and isometric embeddings $\iota_k \colon Z_k \to W$ and $\iota \colon Z_\infty \to W$ with $\iota_n(z_k) = \iota(z_\infty) = q$ such that $\iota_n(Z_k) \to \iota(Z_\infty)$ in the sense of Gromov-Hausdorff convergence and $(\iota_k)_* \mu_k \to \iota_* \mu_\infty$ weakly.
\end{defn}

When the spaces considered are of the form $Z_k = (Z, r_k^{-1}d, z_k)$ and the above limit $(Z_\infty, d_\infty, z_\infty)$ exists, we call this limit a {\em weak tangent} of $Z$.  We assume $Z$ is Ahlfors $Q$-regular.  Then, the weak limit $\mu_\infty$ of $\{r_k^{-Q} \mu \}$ is a measure on $Z_\infty$ such that $(Z_\infty, d_\infty, \mu_\infty)$ is Ahlfors $Q$-regular.  We consider $Z_\infty$ equipped with this measure.

We also need some results relating modulus to discrete modulus.  These results are similar to what appeared in our discussion of the Combinatorial Loewner Property.

\begin{defn}[\cite{Ha}, Definition B.1]
Given a metric space $(Z,d)$ and $K \geq 1$, a cover $\mathscr{S}$ is a {\em $K$-quasi-packing} if for each $s \in \mathscr{S}$ there exists a point $z_s \in s$ and a number $r_s > 0$ such that \\
\indent (i) $B(z_s, r_s) \subseteq s \subseteq B(z_s, K r_s)$ \\
\indent (ii) Every ball $B(z_s, r_s)$ intersects at most $K$ other balls $B(z_s', r_s')$. \\
We call $B(z_s, r_s)$ the {\em internal ball} of $s$ and denote it $B(s)$.
\end{defn}

\begin{lemma}[\cite{Ha}, Lemma B.3]\label{Hai lemma}
Let $(Z, d, \mu)$ be an Ahlfors $Q$ regular metric measure space, $\mathscr{S}$ a $K$-quasi-packing, and $\Gamma$ a family of paths.  Suppose that there exists a constant $\kappa > 0$ for which, for every $s \in \mathscr{S}$ and $\gamma \in \Gamma$, if $\gamma \cap s \neq \emptyset$ then $\diam(\gamma \cap (2K)B(s)) \geq \kappa \diam(B(s))$, where $B(s)$ is the internal ball of $s$.  Then,
\begin{equation*}
\qmod \Gamma \lesssim \qmod(\Gamma, \mathscr{S}).
\end{equation*}
\end{lemma}

Here $\qmod(\Gamma, \mathscr{S})$ is the discrete $Q$-modulus of the curve family $\Gamma$ in the graph induced by the incidence pattern of $\mathscr{S}$; see the discussion about the Combinatorial Loewner Property above for this definition.

\begin{proof}[Proof of Theorem \ref{Wk Tan Thm}]

From \cite[Corollary 3.10]{CP} it follows that there exists a weak tangent $(Z_\infty,d_\infty)$ of $Z$ and two bounded open sets $A = B(z_\infty, 1)$ and $B = B(z_\infty, 3) \setminus \overline{B(z_\infty, 2)}$ in $Z_\infty$ with $\dist(A,B)>0$ such that $\qmod(A,B) > 0$.  By choosing a subsequence if necessary, we may assume the Gromov-Hausdorff distance of $\iota_k(Z_k)$ and $\iota(Z_\infty)$ is $< 1/k$.  We also assume $r_k \to 0$, as otherwise we are dealing with a rescaled metric of our space and the equality of $\ARCdim$ and $Q_w'$ follows from Lemma \ref{modulus Q_w bound}.

Let $X$ be a hyperbolic filling for $Z$ with parameter $s$ and write $X^n$ for the subgraph of $X$ consisting of vertices with level $n$ and edges connecting these vertices.

For $k \in \N$, let $\epsilon = 2/k$.  Let $c = 99/100$.  Define a map from $V_{X^n}$ to subsets of $Z_\infty$ by associating to $v$ the set
\begin{equation*}
S_v = \iota^{-1} (N_{\epsilon} (\iota_k(cB_v))).
\end{equation*}
Let $\mathscr{S}_n = \{S_v : S_v \cap B(z_\infty, 4) \neq \emptyset \}$.  Let $r = r(n,k) = r_k^{-1} 2s^{-n} >0$ be the radius of $B_v$ in the metric of $Z_k$.  Thus, $1/kr = r_k s^n / 2k \to 0$ as $k \to \infty$.  

We show $\mathscr{S}_n$ forms a $K$-quasi-packing of $Z_\infty \cap B(z_\infty, 4)$ for appropriate $n$.  Note $\mathscr{S}_n$ is a cover of $Z_\infty \cap B(z_\infty, 4)$ as any point $z' \in Z_\infty$ has a corresponding point $z \in Z_k$ such that $d_W(\iota_k(z), \iota(z')) < 1/k < \epsilon$ and $\cup B_v$ is a cover of $Z$. 

Given $v \in V_{X^n}$, write $B_v = B(v, 2s^{-n})$.  Define $v' \in S_v$ as follows: by definition, there exists $v_W' \in \iota(Z_\infty)$ such that $d_W(\iota_k(v), v_W') < 1/k$.  Set $v' = \iota^{-1}(v_W')$. 

\sloppypar
Now, let $z' \in B(v', r/8)$.  Let $z \in Z_k$ be such that $d_W(\iota_k(z), \iota(z')) < 1/k$.  Then, 
\begin{equation*}
\begin{split}
d_{Z_k}(v, z) &= d_W(\iota_k(v), \iota_k(z)) \\
&\leq d_W(\iota_k(v), \iota(v')) + d_W(\iota(v'), \iota(z')) +  d_W(\iota(z'), \iota_k(z)) \\
&< 2/k +  r/8.
\end{split}
\end{equation*}
Hence, $z \in (2/kr + 1/8)B_v$.  For large enough $k$ depending on $n$ the ball ${(2/kr + 1/8)B_v}$ is contained in $cB_v$, and hence for such $k$ we have $B(v', r/8) \subseteq S_v$.

\fussy
Now, suppose $z' \in S_v$, so $\iota(z') \in N_\epsilon(\iota_k(cB_v))$.  Thus, there exists $z_W \in \iota_k(cB_v)$ such that $d_W(\iota(z'), z_W) < \epsilon = 2/k$, so 
\begin{equation*}
d_W(\iota(z'), \iota_k(v)) \leq d_W(\iota(z'), z_W) + d_W(z_W, \iota_k(v)) < 2/k + cr.
\end{equation*} 
 As $d_W(\iota_k(v), \iota(v')) < 1/k$ it follows that $d_\infty(z', v') = d_W(\iota(z'), \iota(v')) < 3/k + cr$.  Thus, $S_v \subseteq B(v', Kr)$ where $K = (3/kr) + 1$.  For large enough $k$ we then have $S_v \subseteq B(v', 2r)$.

For condition (ii), suppose $u, v \in V_{X^n}$ and $z' \in B(u', r/8) \cap B(v', r/8)$, where $u'$ is defined as $v'$ is above.  We note $d_{Z_k}(u,v) \geq r/2$ by construction.  We see 
\begin{equation*}
d_W(\iota_k(v), \iota(z')) \leq d_W(\iota_k(v), \iota(v')) + d_W(\iota(v'), \iota(z')) < 1/k + r/8
\end{equation*}
and likewise for $d_W(\iota_k(u), \iota(z'))$.  Hence, $r/2 < 2/k + r/4$ and so $1/4 < 2/kr$ which is impossible for $k$ large enough.  Thus, $B(v', r/8) \cap B(u', r/8) = \emptyset$ for $k$ large.

We conclude $S_v$ forms a $K$-quasi-packing of $Z_\infty \cap B(z_\infty, 4)$ with internal balls $B(v', r/8)$ and $K = 16$. We now look at the incidence pattern of $\mathscr{S}_n$.  Suppose $z' \in S_v \cap S_u$.  Then, $\iota(z') \in N_\epsilon(\iota_k(cB_v)) \cap N_\epsilon(\iota_k(cB_u))$.  Let $z \in Z_k$ be such that $d_W(\iota_k(z), \iota(z')) < 1/k$.  Then, $d_{Z_k}(z, cB_v) < 1/k + \epsilon = 3/k$ and likewise for $cB_u$.  Thus, if $3/k < r/100$ we have the implication $S_v \cap S_u \neq \emptyset \implies B_v \cap B_u \neq \emptyset$.

Note all of the above conclusions were true for $k$ large enough that $1/kr$ is small.  We apply \cite[Lemma B.3]{Ha}, which will give us an upper bound on $k$ depending on $n$.  We note any path $\gamma$ joining $A$ and $B$ in $Z_\infty$ has $\diam(\gamma) > 1/2$.  In \cite[Lemma B.3]{Ha} we have $\diam(B(s)) \leq r/4$ and if $\gamma \cap S_v \neq \emptyset$ and $\diam(\gamma) > r$ then $\diam(\gamma \cap (2K)B(s) ) \geq r$.  Hence, if $1/2 > r$ we may apply \cite[Lemma B.3]{Ha} to the path family $\Gamma$ connecting $A$ and $B$ to conclude that there is a constant $C_{mod} > 0$ such that $C_{mod} < \qmod(\Gamma,\mathscr{S}_n)$.

Let $N \in \N$ and $\theta > 0$.  Recall $r = r_k^{-1} 2s^{-n}$.  We show that we can choose $k$ large enough so that both $1/kr < \theta$ and $r < 1/2$ for at least $N$ values of $n$.  Substituting yields the conditions $r_k s^n / 2 k <\theta$ and $4 < r_k s^n$.  Hence, $s^n \in [4/r_k, 2 \theta k / r_k]$.  Taking logarithms, there are at least $\log_s (\theta k r_k / (8 r_k)) - 1 = \log_s(\theta k / 8) - 1$ such values of $n$, so such a $k$ exists.  In what follows we will use a lower interval bound of $32/r_k$ instead of $4/r_k$, but the same conclusion holds.

We set $V_A^k = \{v : \ell(v) \geq \log_s(32/r_k) \text{ and } d_W(\iota_k(B_v), \iota(A)) < 1/4\},$ set $ A_k = \bigcup_{V_A^k} B_v$, and define $V_B^k$ and $B_k$ similarly.  Let $v \in V_A^k$ and $u \in V_B^k$ and let $z_v \in B_v$ and $z_u \in B_u$.  We estimate $d(z_v, z_u)$.  Note for $n_k = \log_s(32/r_k)$ we have $\diam(\iota_k(B_v)) \leq 2r_k^{-1} 2s^{-n_k}$.  Hence, $d_W(\iota_k(z_v), \iota(A)) \leq 2 r_k^{-1} 2s^{-n_k} + 1/4$.  A similar bound holds for $\iota_k(z_u)$ and $\iota(B)$ and we know $d_W(\iota(A), \iota(B)) \geq 1$.  Hence, $d_W(\iota_k(z_v), \iota_k(z_u)) \geq 1 - 2(r_k^{-1} 4s^{-n_k} + 1/4)$.  It follows from the fact that $d_{Z_k} = r_k^{-1}d$ that $d(z_v, z_u) \geq r_k (1/2 - 8 r_k^{-1} s^{-n_k})$.  This is positive as $s^{n_k} \geq 32 / r_k$.  Now, the minimal diameter of $A_k$ and $B_k$ is bounded below by the diameter of a ball in the smallest level of $V_A^k$ and $V_B^k$, which is comparable to $r_k$ as $s^{n_k} \simeq 32/r_k$ for such balls.  Hence, 
\begin{equation*}
\Delta(A_k, B_k) \lesssim \frac{r_k (1/2 - 8r_k^{-1} s^{-n_k})}{r_k}
\end{equation*}
and so the sets $A_k$ and $B_k$ have controlled relative distance as $k \to \infty$.

Let $a > 0$.  We show that if $k$ is large enough, then $\pcap(A_k, B_k) > a$.  Apply Lemma \ref{Last Path Lemma} with $\delta = 1/5$ and $M$ given by Lemma \ref{Vertex splitting lemma} applied to $X$ to obtain $\beta$ and $\ell_0$.  For $k \in \N$, let $N_k$ be the number of consecutive values $n$ such that $s^n \geq 32/r_k$ and the combinatorial properties derived for $\mathscr{S}_n$ above hold, so $N_k \to \infty$ as $k \to \infty$ by the above.  Call these levels the {\em admissible range} of $k$.  We choose $k$ so that $N_k$ is large enough for the following argument; the threshold for $N_k$ will only depend on $a$, the hyperbolic filling parameters,  $\beta$, and $\ell_0$.

Suppose that $\tau \colon V_X \to [0, 1]$ is a vertex function with $\norm{\tau}_{p,\infty}^p \leq a$.  As $p < Q$, it follows that there exists a constant $C(a)$ such that $\norm{\tau}_Q^Q < C(a)$.  Hence, there exists $L_1 = L_1(a, C_{mod}) \in \N$ such that at most $L_1$ levels $n$ satisfy $\norm{2\tau|_{V_{X^n}}}_Q^Q \geq C_{mod}$.  Similarly, as there exists $L_2 = L_2(a,\beta)$ such that at most $L_2$ levels have a connecting edge $e$ such that $\tau(e) \geq \beta$.  Thus, if $N_k$ is large enough, there exists a level $n_0$ with $\norm{2\tau|_{V_{X^{n_0}}}}_Q^Q < C_{mod}$, the property that any vertex $v$ on $n_0, \dots, n_0 + \ell_0 + 1$ has $\tau(v) < \beta$, and the property that the levels $n_0, \dots, n_0 + \ell_0 + 1$ lie in the admissible range of $k$.

\sloppypar
Define $g \colon \mathscr{S}_{n_0} \to [0,\infty)$ by $g(S_v) = 2\tau(v)$.  Now, as $\norm{2\tau|_{V_{X^{n_0}}}}_Q^Q < C_{mod}$, the function $g$ is not admissible for discrete modulus connecting $A$ and $B$.  Hence, there is a chain of vertices $\hat{\gamma} = \{v_0, \dots, v_J\}$ with $S_{v_0} \cap A \neq \emptyset$, $S_{v_J} \cap B \neq \emptyset$, and $\sum_{j} g(S_{v_j}) < 1/2$.  We claim $B_{v_0} \subseteq A_k$ and $B_{v_J} \subseteq B_k$ as long as $k > 8$.  We prove the first of these claims as the second is similar.  Let $z' \in S_{v_0} \cap A$.  Then, $d_W(\iota_k(B_{v_0}), \iota(A)) \leq d_W(\iota_k(B_{v_0}), \iota(z'))$.  As $z' \in S_{v_0}$, it follows that $d_W(\iota(z'), \iota_k(cB_{v_0})) < 2/k$.  Thus, $d_W(\iota_k(B_{v_0}), \iota(A)) < 2/k < 1/4$ if $k > 8$, so $B_{v_0} \subseteq A_k$.

\fussy

Recall that if $S_v \cap S_u \neq \emptyset$ then $B_v \cap B_u \neq \emptyset$.  Thus, the same path $\hat{\gamma}$ can be traced in $V_{X^{n_0}}$ and hence we can extract from this path tracing a path $\gamma$ connecting $B_{v_0}$ and $B_{v_J}$ such that $\sum_{v \in \gamma} \tau(v) < 1/2$.  We apply Lemma \ref{Last Path Lemma} to find two paths connecting $v_0$ and $v_J$ to $A_k$ and $B_k$ with $\tau$-length $< 1/5$.  Concatenating these paths yields a path connecting $A_k$ to $B_k$ with $\tau$-length $< 1$, so $\tau$ cannot be admissible for $\pcap(A_k, B_k)$.  As $\tau$ with $\norm{\tau}_{p,\infty}^p \leq a$ was arbitrary, it follows that $\pcap(A_k, B_k) \geq a$.  Hence, $\lim_{k \to \infty} \pcap(A_k, B_k) = \infty$.
\end{proof}

\begin{remark}
From Theorem \ref{Wk Tan Thm} we conclude that in this setting $Q_w' \geq Q = \ARCdim$.  Hence, $Q_w' = \ARCdim$ in this case.
\end{remark}


\end{document}